\documentclass[12pt,a4paper,dvipsnames]{amsart}

\usepackage{algorithm,cite,makecell}
\usepackage{algpseudocode}
\usepackage{amsmath, nicefrac, amsthm, verbatim, amsfonts, amssymb, xcolor, bbm}
\usepackage{graphics, xspace, enumerate}
\usepackage[a4paper,margin=3cm]{geometry}
\usepackage{multirow}
\usepackage{setspace}

\usepackage[bb=boondox]{mathalfa}

\usepackage{kbordermatrix}
\usepackage{float}

\usepackage{pgfplots}

\usepackage{graphicx}
\usepackage[colorlinks=true,citecolor=green,urlcolor=green,linkcolor=green,bookmarksopen=true,unicode=true,pdffitwindow=true]{hyperref}
\usepackage[english]{babel}
\usepackage[languagenames,fixlanguage]{babelbib}

\usepackage[labelformat=simple]{subcaption}

\usepackage{tikz}
\usetikzlibrary{graphs,quotes,arrows.meta,bending}

\makeatletter
\@namedef{subjclassname@2020}{\textup{2020} Mathematics Subject Classification}
\makeatother

\theoremstyle{plain}
\newtheorem{theorem}{Theorem}[section]

\newtheorem{lemma}[theorem]{Lemma}
\newtheorem{proposition}[theorem]{Proposition}
\theoremstyle{definition}
\newtheorem{remark}[theorem]{Remark}

\newcommand{\supp}{\operatorname{supp}}

\newcommand{\bbN}{\mathbb{N}}

\newcommand{\aps}[1]{\vert #1 \vert}

\newcommand{\floor}[1]{\lfloor #1 \rfloor}
\newcommand{\FLOOR}[1]{\left\lfloor #1 \right\rfloor}

\newcommand{\OBL}[1]{\left( #1 \right)}

\newcommand{\FJADEF}[3]{{#1}:{#2}\to{#3}}

\numberwithin{equation}{section}

\definecolor{Maroon}{RGB}{140,10,0}

\hypersetup{
	colorlinks,
	linkcolor={Maroon},
	citecolor={Maroon},
	urlcolor={Maroon}
}

\title{On a variant of Flory model}

\author[T.\ Do\v{s}li\'{c}]{Tomislav\ Do\v{s}li\'{c}}
\address[Tomislav\ Do\v{s}li\'{c}]{Department of Mathematics\\
	Faculty of Civil Engineering\\
	University of Zagreb\\ 
	Zagreb\\ 
	Croatia \\ and
Faculty of Information Studies \\
Novo Mesto \\
Slovenia}
\email{tomislav.doslic@grad.unizg.hr}

\author[M.\ Puljiz]{Mate\ Puljiz}
\address[Mate Puljiz]{Department of Applied Mathematics\\
	Faculty of Electrical Engineering and Computing\\
	University of Zagreb\\ 
 Zagreb\\ 
	Croatia}
\email{mate.puljiz@fer.hr}

\author[S.\ \v{S}ebek]{Stjepan\ \v{S}ebek}
\address[Stjepan\ \v{S}ebek]{Department of Applied Mathematics\\
	Faculty of Electrical Engineering and Computing\\
	University of Zagreb\\ 
 Zagreb\\ 
	Croatia}
\email{stjepan.sebek@fer.hr}

\author[J.\ \v{Z}ubrini\'{c}]{Josip\ \v{Z}ubrini\'{c}}
\address[Josip\ \v{Z}ubrini\'{c}]{Department of Applied Mathematics\\
	Faculty of Electrical Engineering and Computing\\
	University of Zagreb\\ 
 Zagreb\\ 
	Croatia}
\email{josip.zubrinic@fer.hr}



\begin{document}
\allowdisplaybreaks[4]

\begin{abstract}

We consider a one-dimensional variant of a recently introduced settlement
planning problem in which houses can be built on finite portions of the 
rectangular integer lattice subject to certain requirements on the amount of
insolation they receive. In our model, each house occupies a unit square
on a $1 \times n$ strip, with the restriction that at least one of the
neighboring squares must be free. We are interested mostly in situations
in which no further building is possible, i.e.\ in maximal configurations
of houses in the strip. We reinterpret the problem as a problem of 
restricted packing of vertices in a path graph and then apply the transfer
matrix method in order to compute the bivariate generating
functions for the sequences enumerating all maximal configurations of a
given length with respect to the number of houses. This allows us to 
determine the asymptotic behavior of the enumerating sequences and to 
compute some interesting statistics. Along the way, we establish close
connections between our maximal configurations and several other types
of combinatorial objects, including restricted permutations and walks on
certain small oriented graphs. In all cases we provide combinatorial proofs.
We then generalize our results in several directions by considering
multi-story houses, by varying the insolation restrictions, and, finally,
by considering strips of width 2 and 3. At the end we comment on several
possible directions of future research.
\end{abstract}

\maketitle

%
%
%
%

\section{Introduction}

Many problems of practical importance can be formulated in terms of packings. 
Intuitively, a packing is any arrangement of non-overlapping copies from a
(usually finite) collection of objects ${\mathcal P}$ within a prescribed
part of a large(r) set ${\mathcal E}$. It often happens that both the large 
set and the objects being packed can be naturally endowed with the same type
of discrete structure. In such cases, general packings can be successfully
modeled by packings of graphs. If, for example, ${\mathcal E}$ can be 
represented by a graph $G$ and elements of ${\mathcal P}$ by graphs $H_1, 
\ldots , H_k$, then a ${\mathcal P}$-packing of $G$ is a collection of
vertex-disjoint subgraphs of $G$ such that each of them is isomorphic to
some $H_i$, $i = 1, \ldots , k$. When ${\mathcal P}$ consists of a single
element $H$, one simply speaks of $H$-packings of $G$.

Clearly, one would expect that the difficulty of packing problems increases
with the increase of the complexity of $H$. Indeed, the simplest non-trivial
case, $H = K_2$, is well researched, while the results on larger $H$ are 
much less abundant. This does not prevent graph packings from being a very
versatile tool; even the simplest case of packing dimers ($H = K_2$) into
a larger graph is one of most commonly used models in several areas of
physics and chemistry. It suffices to mention the Ising model of magnetic
materials and the concept of the topological resonant energy, crucial for
stability of conjugated molecules. Both models employ perfect matchings,
i.e.\ packings of dimers covering all vertices of the underlying graph.
For a very brief introduction to both topics we refer the reader to
\cite[\S 8.7]{LP86} and references therein. For some recent results on packing
larger $H$ see, for example, \cite{doslicASR,doslicPSF}.

In this paper we look at a problem which can be modeled by packing even 
simpler graphs, the copies of $K_1$, into finite portions of regular
rectangular lattice. Without further restrictions this problem would be
trivial, but our problem imposes restrictions that arise quite naturally
in the context of settlement growth and planning. It turns out that with 
those restrictions even packing the simplest possible graphs, $K_1$, into
finite pieces of the square lattice gives rise to very interesting behavior
and exhibits often surprising relations with several other classes of
combinatorial objects. (Another non-trivial problem which can be reduced
to restricted packings of trivial graph $K_1$ is the problem of finding
a large independent set in a given graph.)

All the aforementioned packing problems can be studied in a static or a
dynamic variant. In the present work we focus only on the static models
with the aim to enumerate all configurations that arise in such models
and that satisfy certain additional requirements. The study of the dynamic
variant of the same models aims to find the distribution of configurations
constructed by a random process in which the pieces from $\mathcal{P}$
arrive sequentially and are placed randomly onto available locations in
${\mathcal E}$ until saturation. These kinds of models that evolve over
time are extensively studied, see \cite[\S 7]{Krapivsky_et_al} for
introduction, and \cite{Krapiv20, Krapiv22, Krapiv19, KraRedBird} for some
recent results in this direction. The efforts to extend our results in this
direction are currently underway.

\medskip

In \cite{PSZ-21} three of the present authors introduced the following
settlement model. A rectangular $m\times n$ tract of land, with sides
oriented north-south and east-west, is divided into $mn$ unit squares,
see Figure \ref{fig:tract_of_land}. Each square lot can be either occupied
(by a house) or left vacant. An arrangement of houses on such a tract of
land is called a \emph{configuration} and can be encoded as an $m\times n$
matrix $C$, where $c_{i,j}=1$ if the lot $(i,j)$ is occupied, and
$c_{i,j}=0$ otherwise.

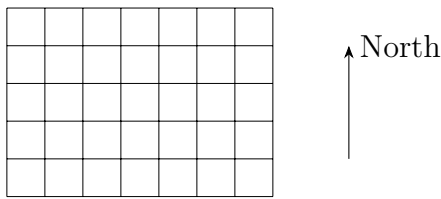
\begin{figure}[h]
	\centering
	\begin{tikzpicture}[scale = 0.5]
	\draw[step=1cm,black,very thin] (0, 0) grid (7,5);
	\draw [-Stealth] (9,1) -- (9,4);
	\node[anchor=west] at (9,4) {North};
	\end{tikzpicture}
	\caption{An example of a tract of land ($m = 5$, $n = 7$).}
	\label{fig:tract_of_land}
\end{figure}

A configuration $C$ is \emph{permissible} if no occupied lot $(i,j)$ borders simultaneously with three other occupied lots to its east, south and west --- in other words --- the house on the position $(i,j)$ receives the sunlight during at least one part of the day (be it in the morning from the east, or during the midday from the south, or in the evening from the west).

As is the case with other packings, one is, naturally, interested in
\emph{large} permissible configurations, since the small ones tend to be
trivial and easy to construct. One way of being large is to have the largest
possible number of occupied lots, hence the largest possible size.
We call such configurations \emph{maximum} configurations (in \cite{PSZ-21,PSZ-21-2} these were called \emph{efficient}).
Another, more interesting, way of being large is in the sense of set inclusion.
A permissible configuration $C$ is \emph{maximal} if no additional houses
can be added to the configuration without rendering it impermissible,
see Figure \ref{fig:examples}. Unlike the maximum configurations, the 
maximal ones usually come in a range of different sizes, and it is of
interest to know the exact distribution of sizes. Moreover, here also the
smallest such configurations are interesting (in \cite{PSZ-21,PSZ-21-2} these were called \emph{inefficient}), as they describe either the
worst possible outcome if we are interested in packing as many elements as
possible, or the best possible outcome if we are trying to satisfy certain
needs by the smallest possible number of packed objects. 

\begin{figure}
	\begin{subfigure}{0.3\textwidth}\centering
		\begin{tikzpicture}[scale = 0.5]
		\draw[step=1cm,black,very thin] (0, 0) grid (4,5);
		\fill[blue!40!white] (0,1) rectangle (1,2);
		\fill[blue!40!white] (0,3) rectangle (1,4);
		\fill[blue!40!white] (0,4) rectangle (1,5);
		\fill[blue!40!white] (1,1) rectangle (2,2);
		\fill[blue!40!white] (1,2) rectangle (2,3);
		\fill[blue!40!white] (1,3) rectangle (2,4);
		\node[] at (1.5,3.5) {x};
		\fill[blue!40!white] (2,0) rectangle (3,1);
		\fill[blue!40!white] (2,1) rectangle (3,2);
		\fill[blue!40!white] (2,2) rectangle (3,3);
		\node[] at (2.5,2.5) {x};
		\fill[blue!40!white] (2,3) rectangle (3,4);
		\fill[blue!40!white] (3,0) rectangle (4,1);
		\fill[blue!40!white] (3,2) rectangle (4,3);
		\draw[step=1cm,black,very thin] (0, 0) grid (4,5);
		\end{tikzpicture}
		\caption{Impermissible}
	\end{subfigure}
	\begin{subfigure}{0.3\textwidth}\centering
		\begin{tikzpicture}[scale = 0.5]
		\draw[step=1cm,black,very thin] (0, 0) grid (4,5);
		\fill[blue!40!white] (0,1) rectangle (1,2);
		\fill[blue!40!white] (0,3) rectangle (1,4);
		\fill[blue!40!white] (0,4) rectangle (1,5);
		\fill[blue!40!white] (1,1) rectangle (2,2);
		\fill[blue!40!white] (1,3) rectangle (2,4);
		\fill[blue!40!white] (2,0) rectangle (3,1);
		\fill[blue!40!white] (2,1) rectangle (3,2);
		\fill[blue!40!white] (2,2) rectangle (3,3);
		\fill[blue!40!white] (2,3) rectangle (3,4);
		\fill[blue!40!white] (3,0) rectangle (4,1);
		\fill[blue!40!white] (3,2) rectangle (4,3);
		\draw[step=1cm,black,very thin] (0, 0) grid (4,5);
		\end{tikzpicture}
		\caption{Permissible}
	\end{subfigure}
	\begin{subfigure}{0.3\textwidth}\centering
		\begin{tikzpicture}[scale = 0.5]
		\draw[step=1cm,black,very thin] (0, 0) grid (4,5);
		\fill[blue!40!white] (0,0) rectangle (1,1);
		\fill[blue!40!white] (0,1) rectangle (1,2);
		\fill[blue!40!white] (0,2) rectangle (1,3);
		\fill[blue!40!white] (0,3) rectangle (1,4);
		\fill[blue!40!white] (0,4) rectangle (1,5);
		\fill[blue!40!white] (1,1) rectangle (2,2);
		\fill[blue!40!white] (1,3) rectangle (2,4);
		\fill[blue!40!white] (2,0) rectangle (3,1);
		\fill[blue!40!white] (2,1) rectangle (3,2);
		\fill[blue!40!white] (2,2) rectangle (3,3);
		\fill[blue!40!white] (2,3) rectangle (3,4);
		\fill[blue!40!white] (2,4) rectangle (3,5);
		\fill[blue!40!white] (3,0) rectangle (4,1);
		\fill[blue!40!white] (3,2) rectangle (4,3);
		\fill[blue!40!white] (3,4) rectangle (4,5);
		\draw[step=1cm,black,very thin] (0, 0) grid (4,5);
		\end{tikzpicture}
		\caption{Maximal}
	\end{subfigure}
	\caption{Examples of impermissible, permissible and maximal configuration on a $5 \times 4$ tract of land. `x' marks a house blocked from the sunlight.}\label{fig:examples}
\end{figure}
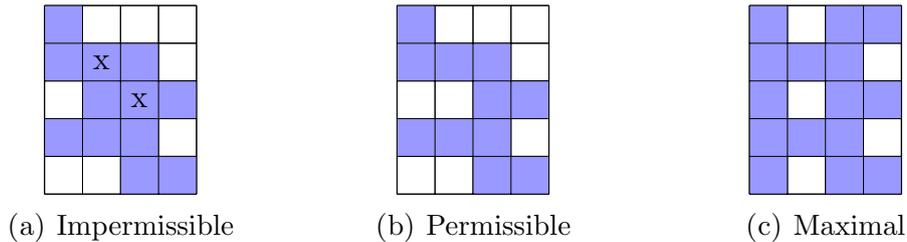

The authors in \cite{PSZ-21} found maximal configurations with the lowest
\emph{occupancy} (the number of houses in a configuration) among all the
maximal configurations on an $m \times n$ grid. They also obtained bounds
on the highest occupancy possible. A natural next step would be
to find the total number of all maximal
configurations for a given grid and to refine the enumeration by the number
of occupied lots. The problem seems to be too hard in the general 
$m \times n$ case. Hence, in this paper we consider its restriction to
the one-dimensional case $1 \times n$ to which we can apply the transfer
matrix method. This method allows us to obtain a complete solution
of the one-dimensional case by computing and analyzing the bivariate
generating functions for the corresponding enumerating sequences.
Our results could be, in principle, generalized to larger grids; indeed,
for grids of size $2\times n$ and $3\times n$
we derive (bivariate) generating functions counting the number of maximal
configurations by the same transfer matrix method used in the $1 \times n$
case. However, the calculations get increasingly infeasible for larger
strips and we decided not to pursue it beyond $m = 3$. 

The transfer matrix method, see \cite[\S 4.7]{Stanley} or
\cite[\S V]{FlajoletSedgewick}, and also \cite[\S 2--4]{SymbDynCoding},
is a well known method for counting words of a regular language.
Applicability of this method to our setting relies on the fact that
permissibility as well as maximality of a configuration can be verified
by inspecting only finite size patches of a given configuration. The
limitation, however, is that the method deals with, essentially, one
dimensional objects, so we first consider a modification of the settlement
model on the $1\times n$ grid. This modification, defined later in text,
we call the Riviera model. It turns out that the Riviera model can be seen
as a variant of Flory polymer model \cite{Flory} which is in turn related
to Page-R\'enyi parking process \cite{GerinPRparkingHAL,Page}.

Surprisingly, the maximal configurations of the Riviera model turn out to be related to a certain kind of restricted permutations introduced in \cite{Baltic}. We were able to construct an explicit bijection translating between the two. Also, we construct another bijection connecting the Riviera model with the closed walks on $P_3$ graph with an added loop. 

The paper is organized as follows. In Section \ref{sec:Riviera} we introduce
the Riviera model. We find the bivariate generating function counting the
number of maximal configurations of length $n$ with precisely $k$ houses
($n, k \in \mathbb{N}$). Furthermore, we relate the Riviera model with some
other combinatorial objects that were already studied in the literature. In
Section \ref{sec:multi-story} we generalize the Riviera model introduced in
Section \ref{sec:Riviera} in the sense that we allow houses to have multiple
stories. Additionally, we comment on the close relation between the Riviera
model and the famous Flory model. In Section \ref{sec:mxnForSmallm} we deal
with configurations on $m \times n$ grids with $m = 2$ and $m = 3$.
Finally, in Section \ref{sec:concluding} we recapitulate our findings and
indicate several
possible directions of future research. Some lengthy formulas are relegated
to Appendix in order to improve legibility.

\medskip

A note on notation: Whenever a non-integer decimal is encountered in the text, its value should be interpreted as an approximation of the true value rounded to six decimal places. By $a_n\sim b_n$ (as $n\to\infty$) we mean $\lim_{n\to\infty} \frac{a_n}{b_n} =1$. In several places in the text we use the same name for different functions. Most prominently, the generating function for almost every model is denoted as $F(x)$, $F(x,y)$, or $F(x,y,z)$. This should not lead to any confusion, as it is always clear from the context to which function the text refers.

\section{Riviera model}\label{sec:Riviera}
We introduce a $1D$-modification of the above settlement planning model which ignores the possibility of obtaining sunlight from the south, but instead retains only the constraints pertaining to the east and west directions. As this is a model on a strip of land, it resembles a Mediterranean settlement along the coast (riviera), hence the name. The configuration of built houses is represented with a row vector\footnote{We write configurations as strings of $0$'s and $1$'s, and we refer to any consecutive sequence of letters in a configuration as a substring or a (sub)word in that configuration.} $C = (c_k)$, where $c_k=1$ if the lot $k$ is occupied and $c_k=0$ otherwise. Similarly as before, a configuration is said to be permissible if every occupied lot has at least one neighboring lot unoccupied (except maybe for the first and the last lot which receive sunlight from the boundary) so that it is not blocked from the sunlight. Among permissible configurations, we are interested in the maximal ones, namely configurations such that any addition of a house on an unoccupied lot would result in an impermissible configuration.

The properties of maximality and permissibility are locally verifiable in a sense that, if one wants to check whether a state of a certain lot (occupied or unoccupied) has caused the configuration to be impermissible or not maximal, one only needs to check the situation on the lots in a certain finite radius of the observed lot, where that radius is uniform for each lot on the tract of land. 

More precisely, to verify that a configuration is permissible, one needs to check that no occupied lot has both of its neighboring lots occupied as well. This can be done by inspecting all the length $3$ substrings of a configuration. And to verify that a configuration is maximal, one needs to check, additionally, that no unoccupied lots can be built on. This can be done by observing the eastern two and western two lots around the unoccupied lot, i.e.\ by inspecting all the length $5$ substrings of a padded (see Remark \ref{rem:padding}) configuration.

This property of local verifiability of a constraint which describes the model is a recurring motif throughout our analysis of related models in this paper.

\begin{lemma}
\label{lemma_forbidden_words}
Let $n\in \mathbb{N}$. A configuration $C \in \{0,1\}^n$ in the Riviera model is maximal if and only if, when padded with zeros, it does not contain any of the following (decorated) substrings: 
\begin{equation}
    1\underline{1}1, \quad 0\underline{0}0, \quad 01\underline{0}0, \quad 0\underline{0}10. 
\end{equation}
\end{lemma}
\begin{remark}\label{rem:padding}
Throughout the paper, unless stated otherwise, we assume that the lots on the boundary can get sunlight from the boundary side, i.e.\ we assume that our configurations are padded with zeros. When inspecting whether a configuration $c_1\dots c_n$ contains a decorated word $d_1\dots \underline{d_k}\dots d_l$, we check against a padded word $\dots000c_1\dots c_n000\dots$ but with the underlined letter of the decorated word aligned with $c_i$ for $i=1,\dots, n$. This is necessary as e.g.\ the configuration $10011$ would otherwise be considered allowed (not containing any of the forbidden substrings), although it is not maximal.
\end{remark}

\begin{proof}[Proof of Lemma \ref{lemma_forbidden_words}]
	A configuration $C$ is maximal if and only if it is permissible and for each $k=1,\dots,n$ one has
	\begin{equation*}
		c_k =0 \Longrightarrow  \begin{array}{cc}
		(c_{k-1}=1 \mbox{ and } c_{k+1}=1) \mbox{ or }  (c_{k-1}=1 \mbox{ and } c_{k-2}=1)  \\
		\mbox{ or } (c_{k+1}=1 \mbox{ and } c_{k+2}=1).
		\end{array}
	\end{equation*}
	The contrapositive of the above implication reads
	\begin{equation}\label{eq:condition1}
	\begin{array}{cc}
	&  (c_{k-1}=0 \mbox{ or } c_{k+1}=0) \mbox{ and }  (c_{k-1}=0 \mbox{ or } c_{k-2}=0)  \\
	& \mbox{ and } (c_{k+1}=0 \mbox{ or } c_{k+2}=0) 
	\end{array}\Longrightarrow c_k =1.
	\end{equation}
	This illustrates the fact that, if there is no danger of losing permissibility by setting $c_k =1$, then one should put $c_k = 1$ (with the agenda of obtaining maximality).
	
	By using the distributive property and after removing redundant terms the left hand side of \eqref{eq:condition1} can be rewritten as
	\begin{equation}\label{eq:condition}
	\begin{array}{cc}
	(c_{k-1}=0 \text{ and } c_{k+1}=0) \text{ or } (c_{k-1}=0 \text{ and } c_{k+2}=0) \\
	\text{ or } (c_{k+1}=0 \text{ and } c_{k-2}=0)
	\end{array}\Longrightarrow c_k =1.
	\end{equation}
	From here we can compile the list of forbidden words. We include $1\underline{1}1$ to ensure permissibility, and \eqref{eq:condition} gives us five more words $0\underline{0}0$, $0\underline{0}{*}0$, $0{*}\underline{0}0$ where $*$ stands for any symbol. As the words $0\underline{0}00$ and $00\underline{0}0$ are already excluded by $0\underline{0}0$, the set of forbidden (decorated) words is
	$$\{1\underline{1}1, 0\underline{0}0, 0\underline{0}10, 01\underline{0}0\}.$$

\end{proof}

\begin{remark}
	An alternative approach for constructing the set of forbidden words, once we know that it suffices checking substrings of length $5$, is to consider all $2^5$ binary words of length $5$ and, out of those, take words that do not appear in any finite maximal configuration to be the forbidden set of words. This approach is more amenable for use in a computer algorithm and we will make use of it later on.
	
	Using this approach one would come up with the set of forbidden length $5$ words
	$$\{{*}1\underline{1}1{*},\quad {*}0\underline{0}0{*},\quad 01\underline{0}00,01\underline{0}01,\quad
	00\underline{0}10,10\underline{0}10\},$$
	which again can be reduced to $\{1\underline{1}1, 0\underline{0}0, 0\underline{0}10, 01\underline{0}0\}$.
As before, $*$ stands for any symbol, and e.g.\ the string ${*}1\underline{1}1{*}$ actually accounts for $4$ different (decorated) words.
\end{remark}
	
\subsection{Counting maximal configurations}
Upon examining the Lemma \ref{lemma_forbidden_words}, one sees that it is possible to encode each maximal configuration as a walk on the directed graph in Figure \ref{fig:transfer_graph_Riviera} whose vertices represent all allowed substrings of length $3$ and the directed edges represent allowed transitions (namely, transitions which comply with the condition stated in Lemma \ref{lemma_forbidden_words}), see \cite[\S 2.3]{SymbDynCoding} for more details on this construction. There is an edge from the word $u_1u_2u_3$ to $v_1v_2v_3$ if they \emph{overlap progressively}, meaning that $u_2u_3=v_1v_2$, and if the word $u_1u_2u_3v_3=u_1v_1v_2v_3$ is not forbidden. 
(Our graph is therefore a subgraph of the $3$-dimensional de Bruijn graph over symbols $\{0,1\}$. Not all edges are present, since the transitions that correspond to forbidden $4$ letter words must be deleted.) Thus, a transition simply represents the addition of a new lot to the right of the configuration, state of which is described with the last letter of the string of the target node. 

	\begin{figure}[h]
		\begin{tikzpicture}[node distance=5em, nodeStyle/.style={draw, circle, minimum size=2.5em}]
		\node (A) [nodeStyle] {100};
		\node (B) [right of = A, nodeStyle] {001};
		\node (C) [right of = B, nodeStyle, fill=blue!20] {011};
		\node (D) [right of = C, nodeStyle, fill=blue!20] {110};
		\node (E) [right of = D, nodeStyle, fill=blue!20] {101};
		\node (F) [right of = E, nodeStyle] {010};
		\draw[-Stealth,above] (A) edge[bend left] (B);
		\draw[-Stealth,above] (B) edge[bend left] (C);
		\draw[-Stealth,above] (F) edge[bend left] (E);
		\draw[-Stealth,above] (D) edge[bend left] (E);
		\draw[-Stealth,above] (D) edge[bend left] (A);
		\draw[-Stealth,above] (C) edge[bend left] (D);
		\draw[-Stealth,above] (E) edge[bend left] (F);
		\draw[-Stealth,above] (E) edge[bend left] (C);
		\end{tikzpicture}
		 \caption{Transfer digraph $\mathcal{G_R}$ for the Riviera model. For example, a maximal configuration 110011010110 is represented with a walk: 	\\	 $ 110\to 100 \to 001 \to 011 \to 110 \to 101 \to 010 \to 101 \to 011 \to 110 $. Each walk must start and end at shaded nodes.}\label{fig:transfer_graph_Riviera}
	\end{figure}
Depending on the choice of boundary conditions, we are left with a constrained subset of vertices which may serve as a starting point or an ending point of the walk which encodes the configuration. The default boundary condition states that the first and the last lot in a configuration obtain sunlight from the boundary. One can easily check that in this situation, the only allowed starting and ending vertices are $110$, $101$, $011$ (Otherwise, one would not obtain a maximal configuration from the walk). An alternative to this boundary condition is the setting where the first and the last lot do not obtain sunlight from the boundary. In this situation, the allowed starting vertices are $100$, $011$, $101$, $010$, while the allowed ending vertices are $001$, $110$, $101$, $010$. A third option which one may consider is the periodic boundary condition in which case, one simply searches for closed walks.

One can count the number of walks of fixed length $n$ on the graph in Figure \ref{fig:transfer_graph_Riviera} by examining the powers of the transfer matrix $A$ associated with that graph: 
\begin{equation}
    \kbordermatrix{
		& \text{100} & \text{001} & \text{011} & \text{110} & \text{101} & \text{010}\\
		\text{100} & 0 & 1 & 0 & 0 & 0 & 0\\
		\text{001} & 0 & 0 & 1 & 0 & 0 & 0\\
		\text{011} & 0 & 0 & 0 & 1 & 0 & 0\\
		\text{110} & 1 & 0 & 0 & 0 & 1 & 0\\
		\text{101} & 0 & 0 & 1 & 0 & 0 & 1\\
		\text{010} & 0 & 0 & 0 & 0 & 1 & 0\\
	} =: A.
\end{equation}
Namely, we have: 
\begin{equation}
\begin{array}{c}
     \# \mbox{ of walks of length $n$}  \\
    \mbox{starting with the node $i$ and ending in the node $j$ }
\end{array}
      = [A^n]_{i,j}.
\end{equation}
This gives us a neat way of counting the maximal configurations since 
\begin{equation}
\begin{array}{c}
     \# \mbox{ of walks of length $n$}  \\
     \mbox{starting with $110$ or $101$ or $011$} \\ 
     \mbox{and ending with $110$ or $101$ or $011$}
\end{array} = \begin{array}{c}
	\# \mbox{ of maximal configurations} \\
	\mbox{of length $n+3$}.
	\end{array}
\end{equation}
This is due to the fact that each transition adds another lot to the configuration, which at the beginning of the walk had $3$ lots.

By introducing the vector $a = (0,0,1,1,1,0)^T$, we have the following: 
\begin{equation}
\label{sequence_an}
    \# \mbox{ of maximal configurations of length $n$} =: a_n = a^T \cdot	A^{n-3}\cdot a, \quad n\ge 3.
\end{equation}

\begin{remark}\label{rem:PF-Riviera}
 A straightforward asymptotic formula for $a_n$ can be obtained by calculating the Perron-Frobenius eigenvalue $\lambda$ of the matrix $A$, namely the largest real eigenvalue of $A$. There exist a constant $C$ such that
 \begin{equation}
     a_n \sim C \, \lambda^n, \text{ as } n \to \infty.
 \end{equation}
 This is due to the fact that vector $a$ possesses a nontrivial component in the direction of the Perron-Frobenius eigenvector. 
 The Perron-Frobenius eigenvalue of our matrix $A$ is equal to $\lambda = \frac{1}{w} = 1.401268$. The numerical value of the constant $C= \frac{\lambda^6+\lambda^5+\lambda^3-\lambda}{2\lambda^4+3\lambda^3+4\lambda^2-6} = 0.803796$ can be obtained from the generating function \eqref{eq:GF4offA080013} using Theorem \ref{tm:asymp} below.
\end{remark}
\begin{remark}\label{rem:differentBoundary}
	 If one would study the alternative boundary conditions of no sun from the boundary, instead of  \eqref{sequence_an}, one would obtain: 
	\begin{equation}
	\label{sequence_bn}
		\begin{array}{cc}
		   \# \mbox{ of maximal configurations of length $n$} \\
		   \mbox{ with no-sun boundary condition }
		\end{array}   =: b_n = b^T \cdot	A^{n-3}\cdot d, \quad n\ge 3,
	\end{equation}
	where $b = (1,0,1,0,1,1)$, $d = (0,1,0,1,1,1)$. This sequence appears on the OEIS \cite{oeis} under the number \href{https://oeis.org/A253412}{A253412}.

	 In the case of periodic boundary conditions, there is a clear $1$ to $1$ correspondence between the maximal configurations of length $n$ and closed walks on the graph in Figure \ref{fig:transfer_graph_Riviera}. Thus, we have: 
	 \begin{equation}
	\label{sequence_dn}
	 \begin{array}{cc}
	      \# \mbox{ of maximal configurations of length $n$}  \\
	      \mbox{ with periodic boundary conditions } 
	 \end{array}   =: d_n = \mathop\text{tr}(A^n) , \quad n\ge 1.
	\end{equation}
	This sequence appears on the OEIS under the number \href{https://oeis.org/A253413}{A253413}.
\end{remark}
	
\subsubsection{Generating functions for the Riviera model}\label{subsec:gen_fun_for_Riviera}
	From the structure of the sequence $(a_n)$, one can easily calculate its generating function $f=f(y)$ by calculating the resolvent $(I-yA)^{-1}$, where $y$ is a formal variable. This is a somewhat standard calculation for which we explicitly need to determine the first $3$ values of $(a_n)$. We have:
	\begin{align}\label{eq:GF4offA080013}\nonumber
	f(y)
	& = 1+y+y^2+ \sum_{n = 3}^{\infty} a^T \cdot
	A^{n-3}\cdot
	a \cdot y^n \\\nonumber
	& = 1+y+y^2+ a^T\cdot \OBL{\sum_{n = 0}^{\infty} (yA)^n} \cdot a  \cdot y^3 \\\nonumber
	& = 1+y+y^2+ a^T \cdot (I - yA)^{-1} \cdot a \cdot y^3\\
	& = \dfrac{1+y+y^3-y^5}{1-y^2-y^3-y^4+y^6}.
	\end{align}
	For inverting matrix functions, we have used the software for symbolic calculation, Maxima \cite{maxima}.

\begin{remark}
 The generating function $f$ encodes the infinite sequence $(a_n)$ into a simple rational function. Immediately we deduce that the sequence $(a_n)$ satisfies the following $6$\textsuperscript{th} order linear recurrence relation: 
 \begin{equation}
     \begin{array}{cc}
          a_n = a_{n-2} + a_{n-3} + a_{n-4} - a_{n-6}, \quad n \geq 7,  \\
          a_1 = 1, \quad a_2 = 1, \quad a_3 = 3, \quad a_4 = 3, \quad a_5 = 4, \quad a_6 = 6.
     \end{array}
 \end{equation}
 The sequence $(a_n)$ cannot easily be represented with an explicit formula as it would involve the roots of the polynomial $p(y) = y^6-y^4-y^3-y^2+1$.
\end{remark}	

The information on the number of maximal configurations of fixed length is already useful, but our aim is to determine the precise number of maximal configurations of length $n \in \mathbb{N}$ with a fixed number of houses $k \in \mathbb{N}$, which we denote by $J_{k,n}$. This information gives us insight into the distribution of the \emph{occupancy} $|C|:= \sum_{i=1}^n c_i$ among maximal configurations of length $n$, for all values of $n \in \mathbb{N}$. Knowing this quantity would lead to determining the so-called \emph{complexity} (see Remark \ref{rem:complexity}) of our model, namely a distribution of occupancy (or associated \emph{building density} which is defined as $\frac{|C|}{n}$) when the length of configurations $n$ grows large.

To this end, we calculate the bivariate generating function $g(x,y)$, where $x$ is a formal variable associated with the occupancy of a configuration, while $y$ remains a formal variable associated with the length of the configuration. We define the following matrix function: 
\begin{equation}
    \kbordermatrix{
		& \text{100} & \text{001} & \text{011} & \text{110} & \text{101} & \text{010}\\
		\text{100} & 0 & x & 0 & 0 & 0 & 0\\
		\text{001} & 0 & 0 & x & 0 & 0 & 0\\
		\text{011} & 0 & 0 & 0 & 1 & 0 & 0\\
		\text{110} & 1 & 0 & 0 & 0 & x & 0\\
		\text{101} & 0 & 0 & x & 0 & 0 & 1\\
		\text{010} & 0 & 0 & 0 & 0 & x & 0\\
	} =: A(x).
\end{equation}
The purpose of this matrix function is to encode when a transition on the graph in Figure \ref{fig:transfer_graph_Riviera} results in the increase of number of occupied lots. Namely:
\begin{equation}
    i \to j \mbox{ is a transition which adds an occupied lot} \iff [A(x)]_{i,j} = x,
\end{equation}
while the rest of the transitions which do not contribute an occupied lot are denoted with $1 = x^0$. The powers of $A(x)$, namely $(A(x))^n$, encode the distribution of occupancies for the configurations of length $n$. We have:
\begin{equation}
    [(A(x))^n]_{i,j} = p_0^{i,j} + p_1^{i,j} x + p_2^{i,j} x^2 + \dots + p_n^{i,j} x^n,
\end{equation}
where 
\begin{equation}
    p_k^{i,j} = \begin{array}{cc}
        \#  \mbox{ of walks of length $n$ on the graph in Figure \ref{fig:transfer_graph_Riviera} }  \\
         \mbox{starting with node $i$ and ending with $j$} \\
         \mbox{where the number of occupied lots was increased by $1$, $k$ times. }
    \end{array}
\end{equation}
In order to take into account the number of occupied lots with which we start the walk, we define vectors: 
\begin{equation}
    a(x) = (0,0,x^2,x^2,x^2,0)^T,\quad  b = (0,0,1,1,1,0)^T.
\end{equation}
By plugging this into the familiar formula and determining the first few terms, we obtain:  
\begin{align*}
	g(x, y)
	& = 1+xy+x^2y^2+ \sum_{n = 3}^{\infty} a(x)^T \cdot
	(A(x))^{n-3}\cdot
	b \cdot y^n \\
	& = \frac{1+xy-(x-x^2)y^2+x^2y^3-x^3y^5}{1-xy^2-x^2y^3-x^2y^4+x^3y^6} \\
	&  = \sum_{n=0}^\infty\sum_{k=0}^\infty J_{k,n}x^ky^n,
\end{align*}
where $J_{k,n}$ is precisely a number of maximal configurations of length $n$ with $k$ occupied lots. 

The bivariate generating function $g(x,y)$ encodes, among other things, the
information on the asymptotic behavior of our sequence $(a_n)$ (the same we
have recovered from the Perron-Frobenius eigenvalue of the transfer matrix
$A$), and also enables us to determine the expected values of two quantities
of interest -- the expected number of buildings in a maximal configuration
of a given length, and the expected length of a maximal configuration with
a given number of buildings. The starting point is the following classical
result, a version of Darboux's theorem as formulated in \cite{bgw}.
For more information on obtaining the asymptotics of a sequence from
its generating function we refer the reader to \cite{bgw,wilf}.

\begin{theorem}\label{tm:asymp}
If the generating function $f(x) =
\sum _{n \geq 0} a_n x^n$ of a sequence $(a_n)$ can be written in the form
$f(x) = \left ( 1 - \frac {x}{w} \right ) ^\alpha h(x)$, where $w$ is the
smallest modulus singularity of $f$ and $h$ is analytic in $w$, then
$a_n \sim \frac {h(w) n ^{- \alpha -1}}{\Gamma (- \alpha) w^n}$, where
$\Gamma $ denotes the gamma function.
\end{theorem}

Now the expected number of built sites in a maximal configuration of length
$n$ can be computed (see \cite{wilf}) as 
$$\frac{[y^n] \frac{\partial g (x,y)}{\partial x} \left|_{x = 1}\right.}{[y^n] g (x,y) \left|_{x=1}\right.},$$
where $[y^n] F(y)$ denotes the coefficient of $y^n$ in the expansion of $F(y)$.

Since $g(1,y) = \frac{p(1,y)}{q(1,y)}$ is rational, its smallest modulus
singularity is the smallest (by absolute value) root of its denominator 
$q(1,y)$. While it does not have a closed-form expression, its approximate
numerical value is readily computed as $w = 0.713639$, the reciprocal
value of the Perron-Frobenius eigenvalue $1.401268$. Now we can write
$$g(x,y)\left |_{x = 1}\right . = \left ( 1-\frac{y}{w} \right )^{-1} g_1(y)$$
and $$\frac{\partial g (x,y)}{\partial x} \left |_{x = 1} \right . = 
\left ( 1-\frac{y}{w} \right )^{-2}g_2(y),$$
where
$$g_1(y) = \frac{p(1,y) (w-y)}{w q(1,y)} \,\,{\rm and}\,\, g_2(y) =
\frac{q(1,y) \frac{\partial p}{\partial x} (x,y) \left |_{x = 1} \right . -
p(1,y) \frac{\partial q}{\partial x} (x,y) \left |_{x = 1} \right .}{\left (w \frac{q(1,y)}{w-y}\right ) ^2}.$$
By evaluating the ratio $r(y) = \frac{g_2(y)}{g_1(y)}$ at $w$ we finally obtain
$$r(w) = \frac{1}{w} \frac{\frac{\partial q}{\partial x} (x,w) \left |_{x = 1}\right .}{\frac{\partial q}{\partial y} (1,y) \left |_{y = w} \right .} = 0.577203$$
and the expected number $\langle k(n) \rangle$ of buildings in a maximal
configuration of length $n$ is given as 
$$\langle k(n) \rangle = 0.577203 \,\,n.$$

If one defines the \emph{efficiency} of a maximal configuration as the ratio
of the actual number of occupied lots and the largest possible number of
occupied lots, which is $\left \lceil  \frac{2n}{3} \right \rceil $, one 
gets the \emph{expected efficiency} of a maximal configuration as
$$\varepsilon = \frac{0.577203 \,\,n}{\left \lceil  \frac{2n}{3} \right \rceil } = 0.865804.$$
This efficiency is higher than the efficiency of unrestricted $P_m$-packings
of $P_n$ for small $m$ (cf. \cite{doslicASR}).

\begin{remark}\label{rem:GFhFixHouses}
 Clearly, by choosing $x=1$ we obtain $f(y)$. However, by choosing $y=1$,
we obtain $h(x)=g(x,1)$ which is a generating function for the sequence
$(h_k)$ which counts the number of maximal configurations with a fixed
number of occupied lots (with variable length). One easily computes this
function as
 \begin{equation}
     h (x) =  \dfrac{1+2x^2-x^3}{1-x-2x^2+x^3}.
 \end{equation}
Straight away we read the recurrence relation for the sequence $(h_k)$:
\begin{equation}
     \begin{array}{cc}
          h_k = h_{k-1} + 2h_{k-2} - h_{k-3}, \quad k \geq 4,  \\
          h_1 = 1, \quad h_2 = 5, \quad h_3 = 5.
     \end{array}
 \end{equation}
By a completely analogous procedure, with switched roles of $x$ and $y$,
one can now compute the expected length $\langle n(k) \rangle$ of a maximal
configuration with $k$ buildings as
$$ \langle n(k) \rangle = 1.758283 \,\,k.$$
We omit the details.

It is, perhaps, interesting to note that $\frac{1}{1.758283} = 0.568737$ which is, as one might hope, close to previously computed $0.577203$. Although this is what one might expect, there is no reason, in general, why should $\frac{\langle n(k)\rangle}{k}$ be equal (or even close) to $\frac{n}{\langle k(n)\rangle}$.

\end{remark}

\begin{remark}\label{rem:complexity}
 	For a building density $\rho\in[0,1]$ one can introduce the number $J_n(\rho) = J_{\floor{\rho n},n}$ which counts the number of configurations of length $n$ with exactly $\FLOOR{\rho n}$ occupied lots. The function 
 	\begin{equation}
 	    f(\rho) = \lim_{n\to \infty}\frac{\ln J_n(\rho)}{n}
 	\end{equation}
 	is called the complexity and it represents the exponential growth rate of the number of configurations with building density $\rho$ as their length $n$ increases. With this function in hand, one can express the asymptotic behavior of the sequence $(J_n(\rho))$ for large $n$ as:
 	\begin{equation}\label{eq:complexity}
 	    J_n(\rho)\sim h(n)\, e^{nf(\rho)} = h(n)\, \left(e^{f(\rho)}\right)^n, \qquad n\to \infty, \quad\rho \in [0,1]
 	\end{equation}
 	where $h(n)$ is some subexponential growing function.
One can easily verify that $\supp f= \left[\frac12,\frac23\right]$. Also, comparing \eqref{eq:complexity} with Remark \ref{rem:PF-Riviera} we infer:
\begin{equation}
    \max_\rho f(\rho) =\ln  \lambda=\ln 1.401268 = 0.337377.
\end{equation}
However, determining the precise formula for $f$ remains an open problem.

In Figure \ref{fig:complexity} the precise values of $J_{k,100}$ are plotted on the log-scale. For large $n$, one expects the shape of this bar plot to approximate the true shape of the complexity function $f(\rho)$.
\end{remark}

\begin{figure}[h!]
    \begin{tikzpicture}
        \begin{axis}
        [ybar,
        bar width=7pt,
        ymode=log,
        xlabel = {$k$},
        ylabel = {$J_{k, n}$}]
        \addplot coordinates {
            (50, 1)
            (51, 40950)
            (52, 47298420)
            (53, 7491483870)
            (54, 308534750280)
            (55, 4489680958620)
            (56, 27525656572050)
            (57, 79341335532896)
            (58, 115332553142708)
            (59, 88281950244176)
            (60, 36391488209400)
            (61, 8109317836050)
            (62, 961479094515)
            (63, 58247672238)
            (64, 1668933267)
            (65, 19597456)
            (66, 70345)
            (67, 34)
            };
        \end{axis}
    \end{tikzpicture}
    \caption{Precise values of $J_{k, n}$ (log-scale) for $n = 100$.}\label{fig:complexity}
\end{figure}
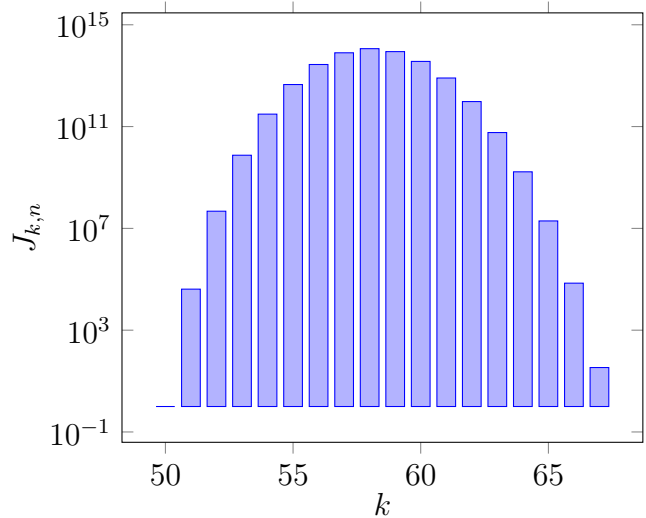

\subsection{Surprising relationship between the Riviera model and other combinatorial objects}

The integer sequence $(a_n)$ associated with the generating function $f=f(y)$ can be evaluated for any $n \in \mathbb{N}$ simply by expanding $f$ into the formal power series in powers of $y$. Even more, by expanding the bivariate generating function $g=g(x,y)$, we obtain the precise distribution of the occupancies of maximal configurations relative to their length. The first few coefficients in the expansion of $g(x,y)$ are given in Table \ref{tab:gf}. By inspecting the non-zero coefficients in the table, we see that the ratio $\frac{k}{n}$ is in-between $\frac{1}{2}$ and $\frac{2}{3}$, for large $n$.
\begin{table}[!h]
	\caption{The first few coefficients in the expansion of the bivariate generating function $g(x,y)$.}\label{tab:gf}
	\small
	\begin{tabular}{c|cccccccccccccccccc}
		$k\backslash n$ & $1$ & $y$ & $y^2$ & $y^3$ & $y^4$ & $y^5$ & $y^6$ & $y^7$ & $y^8$ & $y^9$ & $y^{10}$ & $y^{11}$ & $y^{12}$ & $y^{13}$ & $y^{14}$ & $y^{15}$ & $y^{16}$ & $y^{17}$ \\\hline
		$1$      & $1$ &     &     &     &     &     &     &     &      &      &      &      &      &      &      &      &       &       \\
		$x$      &     & $1$ &     &     &     &     &     &     &      &      &      &      &      &      &      &      &       &       \\
		$x^2$    &     &     & $1$ & $3$ & $1$ &     &     &     &      &      &      &      &      &      &      &      &       &       \\
		$x^3$    &     &     &     &     & $2$ & $3$ &     &     &      &      &      &      &      &      &      &      &       &       \\
		$x^4$    &     &     &     &     &     & $1$ & $6$ & $6$ & $1$  &      &      &      &      &      &      &      &       &       \\
		$x^5$    &     &     &     &     &     &     &     & $3$ & $10$ & $6$  &      &      &      &      &      &      &       &       \\
		$x^6$    &     &     &     &     &     &     &     &     & $1$  & $10$ & $20$ & $10$ & $1$  &      &      &      &       &       \\
		$x^7$    &     &     &     &     &     &     &     &     &      &      & $4$  & $22$ & $30$ & $10$ &      &      &       &       \\
		$x^8$    &     &     &     &     &     &     &     &     &      &      &      & $1$  & $15$ & $49$ & $50$ & $15$ & $1$   &       \\
		$x^9$    &     &     &     &     &     &     &     &     &      &      &      &      &      & $5$  & $40$ & $91$ & $70$  & $15$  \\
		$x^{10}$ &     &     &     &     &     &     &     &     &      &      &      &      &      &      & $1$  & $21$ & $100$ & $168$ \\
	\end{tabular}
\end{table}


The first several values of $(a_n)$ can be read as column sums: $1$, $1$, $3$, $3$, $4$, $6$, $9$, $12$, $16$, $24$, $33$, $46$, $64$, $\dots$ On the other hand, one might do the same for the generating function $h = h(x)$ and the associated sequence $(h_k)$. The first several values of $(h_k)$ can be read as row sums: $1$, $5$, $5$, $14$, $19$, $42$, $66$, $131$, $\dots$ Also, the expansion of $g$ gives insight into the distribution of lengths of maximal configurations relative to the number of occupied lots.


Upon consulting The On-Line Encyclopedia of Integer Sequences \cite{oeis}, we have come across the fact that these sequences were studied in quite different settings (cf.~\href{https://oeis.org/A080013}{A080013} \& \href{https://oeis.org/A096976}{A096976}). We illustrate these connections in the following subsections.

\subsubsection{Bijection with strongly restricted permutations}

The notion of strongly restricted permutations was introduced by Lehmer in \cite{Lehmer}. If $W$ is some fixed subset of integers, one would like to count the number of all the permutations $\pi\in S_n$\footnote{$S_n$ denotes the set of all permutations of the set $[n]=\{1,\dots,n\}$.} such that $\pi(i)-i \in W$, for all $i\in[n]$.

In \cite[Examples 4.7.9, 4.7.17--18]{Stanley} two techniques are presented for obtaining the generating function for the number of strongly restricted permutations for some particular sets $W$, namely the transfer-matrix method and the technique using factorization in free monoids.

In \cite{Baltic} the author devised a new technique for counting restricted permutations in case $\min W =-k$ and $\max W =r$ for some positive integers $k\le r$. When $W=\{-2,-1,2\}$, the sequence counting the corresponding restricted permutations of length $n$ appears in the OEIS under the number \href{https://oeis.org/A080013}{A080013}. The generating function of that sequence is $\dfrac{1-y^2}{1-y^2-y^3-y^4+y^6}$. Note that
$$ \frac{1-y^2}{1-y^2-y^3-y^4+y^6}  = 1+y^3+y^4\cdot \frac{1+y+y^3-y^5}{1-y^2-y^3-y^4+y^6} = 1+y^3+y^4 \cdot f(y),$$
where $f(y)$ is the generating function for the number of Riviera configurations \eqref{eq:GF4offA080013} of fixed length $n$. From here,  the following result is immediate.

\begin{theorem}
	The number of maximal configurations of length $n$ in the Riviera model is equal to the number of permutations $\pi$ of length $n+4$, which satisfy the constraint
	\begin{equation}
	\label{permutationsconstraint}
	    \pi(i)-i\in\{-2,-1,2\}.
	\end{equation}
\end{theorem}

It turns out that one can construct a natural bijection between these two types of objects. The idea is to encode restricted permutations as walks on some digraph, similar to the one in Figure \ref{fig:transfer_graph_Riviera}. If those two graphs are isomorphic, this isomorphism would automatically produce a bijection between the underlying combinatorial objects.

To construct this digraph we, once again, use the transfer-matrix method. One can argue as in \cite[Example 4.7.9]{Stanley} to show that the method is applicable in this case. Let $\pi\in S_n$ be a permutation for which $\pi(i)-i\in W=\{-2,-1,2\}$, for all $i\in[n]$. One can rewrite such a permutation as a sequence of symbols in $W$. In order to check that such a sequence $u_1\dots u_n$ corresponds to a valid permutation, it suffices to check all the substrings of length $5$. This is because the function $\FJADEF{\sigma}{[n]}{[n]}$ defined as $\sigma(i)=i+u_i$ will be a permutation as soon as it is onto; and for this, one only needs to check whether $i\in \{\sigma(i-2), \sigma(i-1), \sigma(i), \sigma(i+1), \sigma(i+2)\}$, for all $3\le i \le n-2$. Additionally, one needs to check that $1$ and $2$ are in the set $\{\sigma(1), \sigma(2), \sigma(3), \sigma(4)\}$, and that $n-1$ and $n-2$ are in the set $\{\sigma(n-3), \sigma(n-2), \sigma(n-1), \sigma(n)\}$. The effect of this being that the walks must start and end at a certain subset of vertices of the constructed digraph. From here, one can write Algorithm \ref{alg:digrafPerm} that produces this digraph which is an induced subgraph of the de Bruijn graph over the set of all $5$ letter words in the alphabet $\{-2,-1,2\}$.

\begin{algorithm}
	\caption{The creation of the digraph $\mathcal{G}$ for strongly restricted permutations}\label{alg:digrafPerm}
	\begin{algorithmic}
		\State AllowedNodes $=\emptyset$
		\State StartNodes $=\emptyset$
		\State EndNodes $=\emptyset$
		\For{$u_1u_2u_3u_4u_5 \in {\{-2,-1,2\}^5}$}
			\If{$3\in \{ 1+u_1,2+u_2,3+u_3,4+u_4,5+u_5 \}$}
			\State add node $u_1u_2u_3u_4u_5$ to AllowedNodes
			\If{$1,2\in \{ 1+u_1,2+u_2,3+u_3,4+u_4,5+u_5 \}$}
			\State add node $u_1u_2u_3u_4u_5$ to StartNodes
			\EndIf
			\If{$4,5\in \{ 1+u_1,2+u_2,3+u_3,4+u_4,5+u_5 \}$}
			\State add node $u_1u_2u_3u_4u_5$ to EndNodes
			\EndIf
			\EndIf
		\EndFor
		\\
		\State $\mathcal{E} = \emptyset$
		\For{$u_1u_2u_3u_4u_5, v_1v_2v_3v_4v_5\in$ AllowedNodes}
		\If{$u_2u_3u_4u_5 = v_1v_2v_3v_4$}
		\State add edge $u_1u_2u_3u_4u_5 \to v_1v_2v_3v_4v_5$ to $\mathcal{E}$
		\EndIf
		\EndFor
		\\
		\State $\mathcal{V} = \emptyset$
		\For{$u_1u_2u_3u_4u_5\in$ AllowedNodes}
		\If{there is a path starting in StartNodes, passing through $u_1u_2u_3u_4u_5$ and ending in EndNodes}
		\State add node $u_1u_2u_3u_4u_5$ to $\mathcal{V}$
		\EndIf
		\EndFor
		\\
		\State remove from $\mathcal{E}$ all the edges not involving nodes in $\mathcal{V}$
		\State\Return $\mathcal{G}=(\mathcal{V},\mathcal{E})$
	\end{algorithmic}
\end{algorithm}

The digraph $\mathcal{G}$ constructed in Algorithm \ref{alg:digrafPerm} has the vertex set $\mathcal{V}$ consisting of $30$ allowed words of length $5$. It turns out that this graph can be further condensed to give a smaller representation of our strongly restricted permutations. If one considers all the $4$-letter words $\{-2,-1,2\}^4$ that do not appear as substrings of the $30$ allowed words, one gets $59$ forbidden words of length $4$. By inspection, one can check that each of the $213=3^5-30$ forbidden $5$-letter words contains one of the $4$-letter forbidden words which means that the same information contained in $\mathcal{G}$ can be encoded in a digraph with a vertex set consisting of only $22=3^4-59$ $4$-letter words. Finally, if we use edges to encode allowed words, rather than just taking the whole induced subgraph of the corresponding de Bruijn graph, we can condense this digraph even further, and obtain the digraph in Figure \ref{fig:amforapermutacije} with 15 nodes representing allowed 3-letter words and an edge from $u_1u_2u_3$ to $v_1v_2v_3$ if and only if $u_1u_2u_3v_3=u_1v_1v_2v_3$ is allowed $4$-letter word. The highlighted nodes are either starting or ending nodes, or, in one case, both.

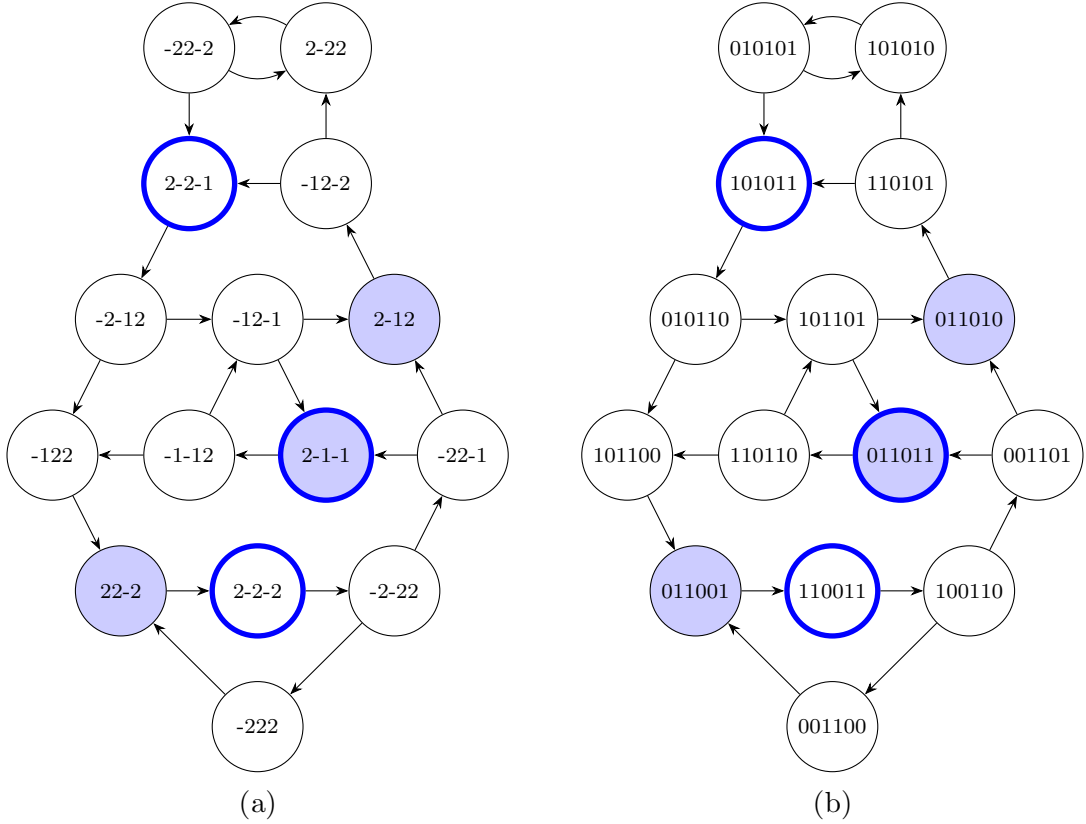
\begin{figure}
	\begin{subfigure}{.495\linewidth}\centering
		\begin{tikzpicture}[scale=1.8, nodeStyle/.style={draw, circle, minimum size=4em}, edgeStyle/.style={-Stealth}]\tiny
		
		\node[nodeStyle] at (0,0) (-222) {-222};
		
		\node[nodeStyle, fill=blue!20] at (-1,1) (22-2) {22-2};
		\node[nodeStyle, draw=blue, line width=2pt] at (0,1) (2-2-2) {2-2-2};
		\node[nodeStyle] at (1,1) (-2-22) {-2-22};
		
		\node[nodeStyle] at (-1.5,2) (-122) {-122};
		\node[nodeStyle] at (-.5,2) (-1-12) {-1-12};
		\node[nodeStyle, fill=blue!20, draw=blue, line width=2pt] at (.5,2) (2-1-1) {2-1-1};
		\node[nodeStyle] at (1.5,2) (-22-1) {-22-1};
		
		\node[nodeStyle] at (-1,3) (-2-12) {-2-12};
		\node[nodeStyle] at (0,3) (-12-1) {-12-1};
		\node[nodeStyle, fill=blue!20] at (1,3) (2-12) {2-12};
		
		\node[nodeStyle, draw=blue, line width=2pt] at (-.5,4) (2-2-1) {2-2-1};
		\node[nodeStyle] at (.5,4) (-12-2) {-12-2};
		
		\node[nodeStyle] at (-.5,5) (-22-2) {-22-2};
		\node[nodeStyle] at (.5,5) (2-22) {2-22};
		
		
		\draw[edgeStyle] (-2-12) to (-122);
		\draw[edgeStyle] (-2-12) to (-12-1);
		\draw[edgeStyle] (-122) to (22-2);
		\draw[edgeStyle] (-12-1) to (2-12);
		\draw[edgeStyle] (-12-1) to (2-1-1);
		\draw[edgeStyle] (-2-22) to (-222);
		\draw[edgeStyle] (-2-22) to (-22-1);
		\draw[edgeStyle] (-222) to (22-2);
		\draw[edgeStyle] (-22-1) to (2-12);
		\draw[edgeStyle] (-22-1) to (2-1-1);
		\draw[edgeStyle] (-1-12) to (-122);
		\draw[edgeStyle] (-1-12) to (-12-1);
		\draw[edgeStyle] (-22-2) to (2-2-1);
		\draw[edgeStyle] (-22-2) to [bend right] (2-22);
		\draw[edgeStyle] (2-2-1) to (-2-12);
		\draw[edgeStyle] (2-22) to [bend right] (-22-2);
		\draw[edgeStyle] (2-2-2) to (-2-22);
		\draw[edgeStyle] (22-2) to (2-2-2);
		\draw[edgeStyle] (2-12) to (-12-2);
		\draw[edgeStyle] (-12-2) to (2-2-1);
		\draw[edgeStyle] (-12-2) to (2-22);
		\draw[edgeStyle] (2-1-1) to (-1-12);
		\end{tikzpicture}
		\caption{}
		\label{fig:amforapermutacije}
	\end{subfigure}
	\begin{subfigure}{.495\linewidth}\centering
		\begin{tikzpicture}[scale=1.8, nodeStyle/.style={draw, circle,minimum size=4em}, edgeStyle/.style={-Stealth}]\tiny
		
		\node[nodeStyle] at (0,0) (001100) {001100};
		
		\node[nodeStyle, fill=blue!20] at (-1,1) (011001) {011001};
		\node[nodeStyle, draw=blue, line width=2pt] at (0,1) (110011) {110011};
		\node[nodeStyle] at (1,1) (100110) {100110};
		
		\node[nodeStyle] at (-1.5,2) (101100) {101100};
		\node[nodeStyle] at (-.5,2) (110110) {110110};
		\node[nodeStyle, fill=blue!20, draw=blue, line width=2pt] at (.5,2) (011011) {011011};
		\node[nodeStyle] at (1.5,2) (001101) {001101};
		
		\node[nodeStyle] at (-1,3) (010110) {010110};
		\node[nodeStyle] at (0,3) (101101) {101101};
		\node[nodeStyle, fill=blue!20] at (1,3) (011010) {011010};
		
		\node[nodeStyle, draw=blue, line width=2pt] at (-.5,4) (101011) {101011};
		\node[nodeStyle] at (.5,4) (110101) {110101};
		
		\node[nodeStyle] at (-.5,5) (010101) {010101};
		\node[nodeStyle] at (.5,5) (101010) {101010};
		
		
		\draw[edgeStyle] (010110) to (101100);
		\draw[edgeStyle] (010110) to (101101);
		\draw[edgeStyle] (101100) to (011001);
		\draw[edgeStyle] (101101) to (011010);
		\draw[edgeStyle] (101101) to (011011);
		\draw[edgeStyle] (100110) to (001100);
		\draw[edgeStyle] (100110) to (001101);
		\draw[edgeStyle] (001100) to (011001);
		\draw[edgeStyle] (001101) to (011010);
		\draw[edgeStyle] (001101) to (011011);
		\draw[edgeStyle] (110110) to (101100);
		\draw[edgeStyle] (110110) to (101101);
		\draw[edgeStyle] (010101) to (101011);
		\draw[edgeStyle] (010101) to [bend right] (101010);
		\draw[edgeStyle] (101011) to (010110);
		\draw[edgeStyle] (101010) to [bend right] (010101);
		\draw[edgeStyle] (110011) to (100110);
		\draw[edgeStyle] (011001) to (110011);
		\draw[edgeStyle] (011010) to (110101);
		\draw[edgeStyle] (110101) to (101011);
		\draw[edgeStyle] (110101) to (101010);
		\draw[edgeStyle] (011011) to (110110);
		\end{tikzpicture}
		\caption{}
		\label{fig:amforakucice}
	\end{subfigure}
	\caption{
		The digraph $\mathcal{G_P}$ in \ref{fig:amforapermutacije} encodes strongly restricted permutations satisfying the constraint \eqref{permutationsconstraint}. The starting nodes are \textcolor{blue!70}{shaded} and \textcolor{blue}{thicker outlines} indicate the ending nodes. The digraph $\mathcal{G_R'}$ in \ref{fig:amforakucice} encodes configurations of the Riviera model using substrings of length $6$. The nodes corresponding to the highlighted nodes in \ref{fig:amforapermutacije} via the unique digraph isomorphism are shaded and outlined in this graph too.}\label{fig:amfora}
\end{figure}

We would now like to match the digraph in Figure \ref{fig:amforapermutacije}, call it $\mathcal{G_P}$, with the digraph in Figure \ref{fig:transfer_graph_Riviera}, call it $\mathcal{G_R}$. Unfortunately, they are not isomorphic, but we can try to create higher edge graphs from the digraph $\mathcal{G_R}$, details below, which encode the same information as $\mathcal{G_R}$ --- in hope of obtaining a graph isomorphic to $\mathcal{G_P}$. This process is opposite of `condensation' we have performed to the digraph produced by the Algorithm \ref{alg:digrafPerm} in order to obtain the digraph $\mathcal{G_P}$.

We have already noted that the graph $\mathcal{G_R}$ shown in Figure \ref{fig:transfer_graph_Riviera} is a subgraph of the $3$-dimensional de Bruijn graph over the alphabet $\{0,1\}$. We can construct a subgraph of the $n$-dimensional de Bruijn digraph, where $n>3$, over the same alphabet, which will encode the same information as $\mathcal{G_R}$. It turns out that $n=6$ will do. The vertex set of this, so called, higher edge graph $\mathcal{G_R'}$ consists of all the allowed words of length $6$, which one can think of as all the possible walks of length $3$ on the graph $\mathcal{G_R}$. A directed edge from $c_1\dots c_6$ to $d_1\dots d_6$ is added to the edge set of $\mathcal{G_R'}$ if and only if the corresponding words overlap progressively ($c_2\dots c_6 = d_1\dots d_5$). The digraph $\mathcal{G_R'}$ is, therefore, the vertex-induced subgraph of the corresponding $6$-dimensional de Bruijn graph. For more details on construction of higher edge graphs, see \cite[Definition 2.3.4]{SymbDynCoding}.

The graph $\mathcal{G_R'}$ obtained by the above procedure, is shown in Figure \ref{fig:amforakucice}. Note that it is isomorphic to $\mathcal{G_P}$, and that this isomorphism is unique. Also note that the set of nodes at which the walks on $\mathcal{G_R'}$ would be allowed to start and end is much larger than the set highlighted in Figure \ref{fig:amforakucice}. More precisely, any node $c_1\dots c_6$ for which $c_1c_2c_3 \in \{110,101,011\}$ would be a starting node, and if $c_4c_5c_6 \in \{110,101,011\}$, it would be an ending node. But the walks of length $n+1$ on $\mathcal{G_R'}$ would then account for all the maximal configurations in the Riviera model of length $n+7$ --- and that is not what we want, since the walks of length $n+1$ on $\mathcal{G_P}$ encode the strictly restricted permutations of $[n+4]$.

If we consider the walks on $\mathcal{G_R'}$ which start and end at the nodes that correspond to starting and ending nodes in $\mathcal{G_P}$, we immediately note that all the configurations obtained in such a way always start with $0110$ and end with $011$. Using the graph $\mathcal{G_R}$ in Figure \ref{fig:transfer_graph_Riviera} it is clear that adding the prefix $0110$ and suffix $011$ to a maximal configuration, again produces a 7-blocks longer (permissible) maximal configuration. This is because from each starting node, there is a backward path (going along edges in the direction opposite to the arrow direction) of length $4$ which produces the prefix $0110$; also from each ending node, there is a $3$-step continuation of path which produces the suffix $011$. Conversely, removing that same prefix and suffix from a maximal configuration of length $n+7$, produces a maximal configuration of length $n$. We can again argue using the graph $\mathcal{G_R}$. Any walk starting with $011 \to 110$ after three steps must again reach one of the starting nodes; and walk ending in $011$ when traced backwards must, after three steps going backwards, reach one of the ending nodes. This shows that there is a bijective correspondence between all the maximal Riviera configurations of length $n$ and the maximal Riviera configurations of length $n+7$ starting with $0110$ and ending with $011$ which in turn are in a bijective correspondence with the strongly restricted permutations of length $n+4$. The bijection is obtained by translating walks on $\mathcal{G_P}$ to walks on $\mathcal{G_R'}$ and the other way around.

It is, in fact, possible to specify this bijection even more concisely, circumventing the graphs in Figure \ref{fig:amfora} altogether. Compare each edge in $\mathcal{G_R}$ with all its associated edges in $\mathcal{G_R'}$ and note that the corresponding edges in graph $\mathcal{G_P}$ all represent adding the same symbol at the end. E.g.\ the transition $011 \to 110$ in $\mathcal{G_R}$ corresponds to transitions $101011 \to 010110$, $011011 \to 110110$ and $110011 \to 100110$ in $\mathcal{G_R'}$ and all of them in $\mathcal{G_P}$ correspond to adding the letter $2$ at the end. Collecting all this information together, we can label the edges of the graph $\mathcal{G_R}$ in Figure \ref{fig:transfer_graph_Riviera} with the appropriate letter which is being added in the permutation graph $\mathcal{G_P}$ corresponding to that transition. This edge-labeled graph is given in Figure \ref{fig:graf_za_bijekciju}.

We now summarize how to bijectively map any maximal Riviera configuration of length $n$ to a strongly restricted permutation of length $n+4$ using Figure \ref{fig:graf_za_bijekciju}. Take any such maximal configuration and prefix it with $0110$ and suffix it with $011$. Then take a walk over the graph in Figure \ref{fig:graf_za_bijekciju} (which will be of length $n+7-3=n+4$) and collect the labels $u_1\dots u_{n+4}$ of all the edges traversed. Finally, construct the bijection $\FJADEF{\sigma}{[n+4]}{[n+4]}$ as $\sigma(i)=i+u_i$ for $i\in [n+4]$.

As an example, the maximal configuration $10110$ is first enlarged to the maximal configuration $\textcolor{blue}{0110}|10110|\textcolor{blue}{011}$. Next, we examine the unique walk determined by this configuration:  $\text{011} \stackrel{2}{\longrightarrow} \text{110} \stackrel{-1}{\longrightarrow} \text{101} \stackrel{2}{\longrightarrow} \text{010} \stackrel{-2}{\longrightarrow} \text{101} \stackrel{-1}{\longrightarrow} \text{011} \stackrel{2}{\longrightarrow} \text{110} \stackrel{2}{\longrightarrow} \text{100} \stackrel{-2}{\longrightarrow} \text{001} \stackrel{-2}{\longrightarrow} \text{011}$.  This walk generates the permutation $\sigma$ encoded with the string 2-12-2-122-2-2, which is the permutation $\begin{pmatrix}
1 & 2 & 3 & 4 & 5 & 6 & 7 & 8 & 9 \\
3 & 1 & 5 & 2 & 4 & 8 & 9 & 6 & 7 
\end{pmatrix}$.

\begin{figure}
	\begin{tikzpicture}[node distance=5em,nodeStyle/.style={draw, circle, minimum size=2.5em}]
	\node (A) [nodeStyle] {100};
	\node (B) [right of = A, nodeStyle] {001};
	\node (C) [right of = B, nodeStyle, fill=blue!20] {011};
	\node (D) [right of = C, nodeStyle, fill=blue!20] {110};
	\node (E) [right of = D, nodeStyle, fill=blue!20] {101};
	\node (F) [right of = E, nodeStyle] {010};
	\draw[-Stealth,above] (A) edge[bend left] node {-2} (B);
	\draw[-Stealth,above] (B) edge[bend left] node {-2} (C);
	\draw[-Stealth,above] (F) edge[bend left] node[below] {-2} (E);
	\draw[-Stealth,above] (D) edge[bend left] node {-1} (E);
	\draw[-Stealth,above] (D) edge[bend left] node[below] {2} (A);
	\draw[-Stealth,above] (C) edge[bend left] node {2} (D);
	\draw[-Stealth,above] (E) edge[bend left] node {2} (F);
	\draw[-Stealth,above] (E) edge[bend left] node[below] {-1} (C);
	\end{tikzpicture}
    \caption{The digraph $\mathcal{G_R}$ with labeled edges which encodes the bijection between maximal configurations of length $n$ in the Riviera model and strongly restricted permutations of length $n+4$ where $W=\{-2,-1,2\}$.}
    \label{fig:graf_za_bijekciju}
\end{figure}
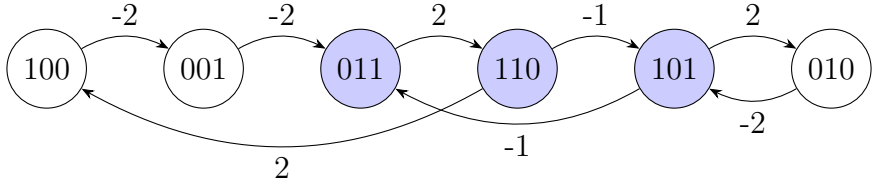

We end this section with a remark which will prove useful in the next subsection.
\begin{remark}\label{rem:1001}
	Above, we have argued that taking any maximal configuration $c_1 \dots c_n$ and prefixing it with $0110$ and suffixing it with $011$ yields a bijection between all the maximal Riviera configurations of length $n$ and the maximal Riviera configurations of length $n+7$ starting with $0110$ and ending with $011$.
	
	If we further add prefix $10$ and suffix $001$ to these already extended configurations, we obtain a bijective correspondence between the maximal Riviera configurations $c_1 \dots c_n$ of length $n$ and the configurations of length $n+12$ starting with $100110$ and ending with $011001$ which, although not maximal (because of the boundary condition), do not contain\footnote{Here, and throughout this section, by contained we mean contained as a substring in unpadded configurations.} any substrings forbidden by Lemma \ref{lemma_forbidden_words}. Each of those extended configurations can, therefore, be represented as a walk on $\mathcal{G_R}$ (Figure \ref{fig:transfer_graph_Riviera}) starting at the node $100$ and ending at the node $001$. Conversely, if a configuration $100110 c_1\dots c_n 011001$ (of length $n+12$) does not contain any substrings forbidden by Lemma \ref{lemma_forbidden_words}, or equivalently, can be represented as a walk on $\mathcal{G_R}$ (of length $n+9$) starting at $100$ and ending at $001$, then after removing prefix $ 100110$ and suffix $011001$ one is left with a proper maximal configuration $c_1\dots c_n$ of length $n$. This is because removing prefix and suffix corresponds to cutting off the first part of the walk $100$-$001$-$011$-$110$-$10c_1$-$0c_1c_2$ and the last part of the walk $c_{n-1}c_n0$-$c_n01$-$011$-$110$-$100$-$001$. Note that regardless of what $c_1$, $c_2$, $c_{n-1}$, and $c_n$ are --- the next node after $0c_1c_2$ as well as the node just before $c_{n-1}c_n0$ will always have to be one of the starting/ending nodes, which means that the remaining part of the walk encodes a proper maximal configuration $c_1c_2\dots c_{n-1}c_n$.
\end{remark}

\subsubsection{Bijection with the closed walks on the $P_3$ with a loop}

\begin{figure}
	\begin{subfigure}{.49\linewidth}\centering
		\begin{tikzpicture}[node distance=5em,every loop/.style={}]
		\node (A) {\bf\color{red}\Large$\bullet$};
		\node (B) [right of = A] {$\bullet$};
		\node (C) [right of = B] {$\bullet$};
		\draw[above] (A) edge (B);
		\draw[above] (B) edge (C);
		\draw (C) edge[loop,out=45,in=-45,looseness=8] (C)[below];
		\end{tikzpicture}
		\caption{$P_3$ with a loop, $\mathcal{P}$ for short}\label{fig:p3withloop}
	\end{subfigure}
	\begin{subfigure}{.49\linewidth}\centering
		\begin{tikzpicture}[node distance=5em,every loop/.style={}]
		\node (A) {\bf\color{red}1001};
		\node (B) [right of = A] {11};
		\node (C) [right of = B] {101};
		\draw[above] (A) edge (B);
		\draw[above] (B) edge (C);
		\draw (C) edge[loop,out=45,in=-45,looseness=8] (C)[below];
		\end{tikzpicture}
		\caption{$\mathcal{P}$ with labeled nodes}\label{fig:LabeledP_3}
	\end{subfigure}
	\caption{}
\end{figure}

In Remark \ref{rem:GFhFixHouses} we have derived the generating function $h(x)$ for the number of Riviera configurations (of variable length) containing a fixed number $k$ of occupied lots. This sequence appears on OEIS \cite{oeis} in two instances as \href{https://oeis.org/A052547}{A052547} with offset 3 and as \href{https://oeis.org/A096976}{A096976} with offset 5. There are three more related sequences: \href{https://oeis.org/A006053}{A006053}, \href{https://oeis.org/A028495}{A028495}, and \href{https://oeis.org/A096975}{A096975}, satisfying the same recurrence relation with different initial conditions. Each of these sequences is connected to the number of walks on the graph $P_3$ (the path graph over three nodes) with a loop added at one of the end nodes. This graph is represented in Figure \ref{fig:p3withloop} and we denote it by $\mathcal{P}$. The precise connection relating this graph with our sequence is given in the following theorem.

\begin{theorem}
	The number of maximal Riviera configurations containing exactly $k$ occupied lots is equal to the number of closed walks of length $k+4$ on the graph $\mathcal{P}$ which start and end at the node of degree $1$. There is a natural bijection relating these quantities.
\end{theorem}
\begin{proof}
	From Lemma \ref{lemma_forbidden_words} we know that no three consecutive $0$'s are allowed in a maximal configuration. That means that each two neighboring $1$'s must be separated by zero, one or two $0$'s. This further means that, after ignoring leading and trailing $0$'s each maximal configuration can be identified by a sequence of strings in the set $\{11, 101, 1001\}$. We assume here that the last $1$ in one string overlaps with the first $1$ in the next. E.g.\ we would split the configuration $11011001101$ as $11$-$101$-$11$-$1001$-$11$-$101$.
	
	From Lemma \ref{lemma_forbidden_words} we also see that $11$ cannot be followed or preceded by $11$ (as this would produce $111$); $101$ and $1001$ cannot be followed or preceded by $1001$ (as this would produce $0100$ or $0010$). It is easy to see that the remaining transitions: $1001$-$11$ and $11$-$101$ going in either direction, and the loop at $101$ --- can all appear in a maximal configuration and are, thus, all allowed. These transitions are shown in the node-labeled graph $\mathcal{P}$ in Figure \ref{fig:LabeledP_3}. We use undirected edges as in each case the transitions going either way are allowed.
	
	Consider now the mapping which to each maximal Riviera configuration $c_1\dots c_n$ assigns the configuration $100110c_1\dots c_n011001$. By Remark \ref{rem:1001} we know that this map is a bijection from the set of all maximal Riviera configurations with exactly $k$ occupied lots to the set of configurations of the form $100110c_1\dots c_n011001$ which have exactly $k+6$ occupied lots. Those obtained configurations are not maximal but do not contain substrings forbidden by Lemma \ref{lemma_forbidden_words}.
	
	Now each of those configurations of the form $100110c_1\dots c_n011001$, where $c_1\dots c_n$ is a proper maximal configuration with $k$ occupied lots, can be represented as a walk of length $k+4$ on the graph $\mathcal{P}$ in Figure \ref{fig:LabeledP_3} which starts and ends at $1001$. Conversely, one easily checks (by inspecting all length $2$ walks) that a walk on this graph can never produce a configuration containing a substring which is forbidden by Lemma \ref{lemma_forbidden_words}. Therefore, each walk of length $k+4$ starting and ending at $1001$ will necessarily produce a configuration of the form $100110c_1\dots c_n011001$ which does not contain a substring forbidden by Lemma \ref{lemma_forbidden_words} and has $k+6$ $1$'s. By Remark \ref{rem:1001}, the word $c_1\dots c_n$ will be a proper maximal configuration with exactly $k$ $1$'s.
	
	Putting everything together gives us the required bijection. A maximal Riviera configuration $c_1 \dots c_n$ containing exactly $k$ occupied lots is written as the string $100110c_1\dots c_n 011001$, which is then represented as a walk of length $k+4$ over the graph $\mathcal{P}$. As an example, the maximal configuration 10110 is mapped to $10110 \to \textcolor{blue}{100110}|10110|\textcolor{blue}{011001}$ which corresponds to the walk: $1001 \rightarrow 11 \rightarrow 101 \rightarrow \text{101} \rightarrow \text{11}
	\rightarrow \text{1001} \rightarrow \text{11} \rightarrow \text{1001}$ which begins and ends with the node $1001$.
\end{proof}

\section{Multi-story models}\label{sec:multi-story}
A natural generalization of the Riviera model is an analogous model where the lots are occupied with houses consisting of multiple stories. The configurations of houses in this case are represented with a row vector $C=(c_1,\dots,c_n) \in \mathbb{N}_0^n$, where $c_i$ is the number of stories that the house on the lot $i$ has, and $c_i = 0$ represents an empty lot, as before.

For the sake of simplicity, the assumption is that the sunlight falls at the angle of $45$ degrees and that each floor of every house is a perfect cube, see Figure \ref{fig:sjene}. This assumption requires the building configurations to spread out more so that the lower stories can obtain sunlight and not be blocked by other houses. Thus, the following definition arises: A building of $k$ stories positioned on the lot $i$ blocks sunlight, from one side, to each lot $ i-k \leq j \leq i+k$ up to the height of $k - \aps{i - j} + 1$ stories.

\begin{figure}
	\includegraphics{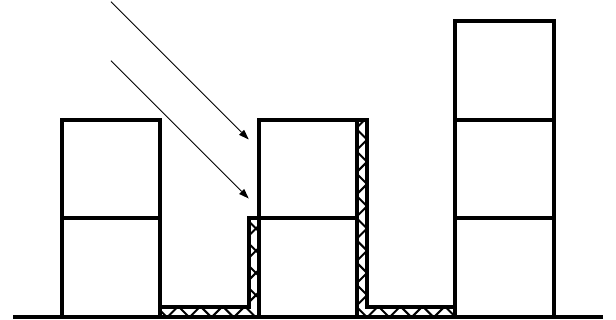}
	\caption{Both stories of the middle house are blocked from the sunlight from the east, while only the first story is blocked from the west.}\label{fig:sjene}
\end{figure}

The permissibility constraint in such models can be concisely stated in the following way: Every story of each house must obtain sunlight from at least one side (east or west) during the day, namely, not to be blocked from sunlight from both sides by other buildings. Naturally, we are again interested in maximal configurations. The maximality of the configuration $C \in \mathbb{N}_0^n$ would mean that for each lot $i$, $1\leq i \leq n$, no additional stories can be built on it. In other words, if an addition of a single story to any lot in $C$ would still result in a permissible configuration, then it is not maximal. 

In this section we also comment a logical counterpart of a multi-story Riviera model, namely a multi-story Flory model. As already stated, in the Riviera model we require that each story obtains sunlight from the east \textbf{or} the west.  However, one could alternatively require that each story obtains sunlight from the east \textbf{and} from the west. In this case, we would obtain models of Flory type, which in the single-story case comes down to the classical Flory model introduced in \cite{Flory}. We comment on both of these variants, which we can refer to as the \textbf{AND} and the \textbf{OR} variant.

\subsection{OR variant}
Our intention here is again to count the number of maximal configurations. To this end, we apply the transfer matrix method, for which we are required to construct a digraph which allows us to encode our configurations in terms of walks on that graph. It is somewhat obvious that, on this level of generality, this is impossible, due to the fact that the minimal set of impermissible substrings of maximal configurations cannot be reduced to a finite set. However, by fixing a maximal number of stories allowed, we denote it with $k$, we are still in the domain of transfer matrix method. 

Once we fix $k$, in order to construct a digraph, we must know the minimal length of the nodal strings which would guarantee the permissibility and maximality of the configurations associated with walks. Next, we must find all of the possible substrings of that length which occur in maximal configurations. One way of doing this is by an exhaustive search among all the maximal configurations until one is sure that all the substrings have been found. 

This, however, is quite involved, which can be deduced from the following two lemmas.

\begin{lemma}\label{lm:4k+1}
Let $k \in \mathbb{N}$ be the maximal number of stories allowed. It is enough to take substrings of length $4k + 1$ to form nodes of the graph $\mathcal{G}=(\mathcal{V},\mathcal{E})$ (which can be constructed analogously as in Algorithm \ref{alg:digrafPerm}) to be certain that this graph $\mathcal{G}$ can be used in the transfer matrix method for counting the corresponding maximal configurations.
\end{lemma}
\begin{proof}
As already mentioned, applicability of the transfer matrix method relies on the fact that permissibility, as well as maximality of a configuration, can be verified by inspecting only finite size patches of a given configuration. More precisely, if one wants to check whether a state of a certain lot (occupied or unoccupied) has caused the configuration to be impermissible or not maximal, one only needs to check the situation on the lots in a certain finite radius of the observed lot, where that radius is uniform for each lot on the tract of land. What we are claiming in this lemma is that, when the maximal number of stories allowed is $k$, than this radius is $2k$, i.e.\ it is enough to check the $2k$ immediate neighbors to the west of a lot and the $2k$ immediate neighbors to the east of a lot to be certain whether this central lot violates permissibility or maximality.

Let us first discuss what needs to be checked to be sure that building an additional story on that central lot will not violate permissibility and then we will discuss what needs to be checked to ensure that not building an additional story on this lot will not violate maximality. There are two ways in which permissibility can be violated. On one hand, it can happen that the built story itself will not be exposed to sunlight and on the other hand, the built story can block a story of another house from the sunlight. The radius that needs to be checked is the biggest in the case when we want to build as many stories as allowed, so we assume from now that we want to build a $k$-story house and that its first neighbors to the east and to the west are also $k$-story houses. To check whether the central lot is already blocked from the sunlight, we only need to check $k$ lots to the east and $k$ lots to the west. Notice now that if the first neighbor to the east of the central lot is more than $k$ empty lots away, then this neighbor is definitely getting sunlight from the west and building a house on a central lot will not block the sunlight to that neighboring house. If that neighbor is less than $k$ empty lots away, then building a $k$-story house on the central lot would block the sunlight from the west at least to the first story of the mentioned neighbor. In that case, the question is whether this first neighbor to the east is exposed to sunlight from the east. To verify that, we again need to check what happens on $k$ immediate neighboring lots to the east of it. Hence, the worst-case scenario is that we need to check $2k$ lots to the east of the central lot. Due to the symmetry, the same holds while checking lots to the west of the central one.

It is now easy to see that applying the same logic and checking $2k$ lots immediately to the east and immediately to the west of a central lot is enough to verify whether we can build additional stories on that central lot and in this way we check maximality.
\end{proof}

\begin{lemma}\label{lm:8k+1}
	Let $k \in \mathbb{N}$ be the maximal number of stories allowed. Then the maximal configurations of length $8k + 1$ contain as substrings all the strings of length $4k + 1$ that can appear in any maximal configuration of arbitrary length.
	
	Moreover, if such a string is at the beginning/end of a maximal configuration then it is possible to choose a maximal configuration of length $8k+1$ in which it also appears at the beginning/end.
\end{lemma}
\begin{remark}
	Notice that Lemma \ref{lm:8k+1} tells us that if we find all the maximal configurations of length $8k + 1$ (which can be done by exhaustive search using a computer), and then go through those maximal configurations with the window of size $4k + 1$, we will find all the vertices that we need for the graph $\mathcal{G}$. Moreover, using these maximal configurations of length $8k + 1$ we can easily extract the vertices with which the maximal configuration can start and end (see Algorithm \ref{alg:cap}).
\end{remark}
\begin{proof}[Proof of Lemma \ref{lm:8k+1}]
	It is clear from Lemma \ref{lm:4k+1} that the state of a certain lot in a maximal configuration, padded with zeros, is a function of the states of $2k$ lots immediately to the east and $2k$ lots immediately to the west of it. More precisely, the state $s$ of that lot is the largest number $0\le s \le k$ which keeps the configuration permissible. This means that it is impossible to have two nodes in the graph $\mathcal{G}$ such that they only differ in the state of the central lot.
	
	The above mentioned function $f({u_1\dots u_{2k}},{v_1\dots v_{2k}})$, which to each pair of permissible words of length $2k$ assigns the largest number of stories that could be built on the central lot of a configuration obtained by concatenating the word $u_1\dots u_{2k}$ to the left and the word $v_1\dots v_{2k}$ to the right of that central lot whilst not violating permissibility, is monotonic in the sense that if $u_i\le u_i'$ and $v_i\le v_i'$ for all $1\le i\le 2k$ then $f({u_1\dots u_{2k}},{v_1\dots v_{2k}}) \ge f({u_1'\dots u_{2k}'},{v_1'\dots v_{2k}'})$.
	
	What we need to show here is that every substring of length $4k + 1$ that can appear in some maximal configuration of arbitrary length, also appears in some maximal configuration of length $8k + 1$. We prove this by showing that any substring of length $4k + 1$ that can appear in a maximal configuration can be expanded to a maximal configuration by adding additional $2k$ lots to the east and to the west of it and building additional houses only on those $4k$ added lots.
	
	As announced, we start from a substring $u_1 u_2 \dots u_{4k+1}$ of length $4k + 1$ which appears in some maximal configuration. We then extract this substring along with $2k$ lots to the east of that substring within this maximal configuration in which it appears and also $2k$ lots to the west of it, padding with zeros when necessary. This gives us a string $u_{-2k+1} \dots u_0 u_1 \dots u_{4k+1} \dots u_{6k+1}$ of length $8k + 1$ which we further pad with zeros to the left and right to obtain an infinite word $(u_i)_{i\in\mathbb{Z}}$.
	
	Note that the maximality of the initial configuration in which the substring $u_1 u_2 \dots u_{4k+1}$ appeared implies that the relation $u_i=f(u_{i-2k}\dots u_{i-1},u_{i+1}\dots u_{i+2k})$ holds for any $1\le i\le 4k+1$.
	
	We now start changing the entries $u_0, u_{-1}, \dots , u_{-2k+1}$, and $u_{4k+2}, u_{4k+3}, \dots , u_{6k+1}$ one-by-one, in this order, according to the output of $f$ on the current state of the $2k$ lots to the left and right of the observed block. We first change $u_0$ to $u_0'=f({u_{-2k}\dots u_{-1}}, {u_1 u_2 \dots u_{2k}})$, then $u_{-1}$ to $u_{-1}'=f({u_{-2k-1}\dots u_{-2}}, {u_0' u_1 \dots u_{2k-1}})$ and so on. Continuing in this way, we will construct the configuration $$u_{-2k+1}'\dots u_{0}' u_1 u_2 \dots u_{4k+1} u_{4k+2}' \dots u_{6k+1}'$$ of length $8k+1$ which, we claim, is maximal.
	
	Observe that since we have started with a permissible configuration, the updates provided  by $f$ can only ever increase the number of the stories already present at the inspected lot. Also note that the configuration we obtain after performing these updates is also permissible as the updates provided by the function $f$ never violate permissibility.
	
	Now, to verify maximality note that, by the definition of function $f$, the number $u_{i}'$, where $-2k+1\le i\le 0$ or $4k+2\le i\le 6k+1$, was as large as permitted at the time the update $u_i \mapsto u_i'$ was performed. Taking into account the monotonicity of $f$, it is clear that $u_i'$ is as large as possible in the final configuration obtained. Finally note that, again, because of the monotonicity of $f$, $u_i$ for $1\le i\le 4k+1$ is still as large as possible.
	
	To see that the `moreover' part of the statement is true, note that an analogous procedure can be followed if the initial substring appears at the beginning/end. In this case we extract the substring from its maximal configuration along with $4k$ additional lots to its right/left and then pad with zeros. The updates via $f$ are again performed only on these extra lots and similar arguments show that the obtained configuration of length $8k+1$ which starts/end with the chosen substring is maximal.
\end{proof}

Now we propose an algorithm (see Algorithm \ref{alg:cap}) for calculating the $(k+1)$-variate generating function related to a Riviera model where it is possible to have houses with different number of stories, but the maximal number of stories allowed is a fixed number $k \in \mathbb{N}$. Variables $x_r$, $r \in \{1, 2, \ldots, k\}$, are formal variables associated with the number of $r$-story houses in the maximal configuration and $y$ is a formal variable associated with the length of a maximal configuration. The proposed algorithm is written in a very general way and it works even in situations where all the $r$-story houses appear ($r \in \{1, 2, \ldots, k\}$), but in the case when we allow only $k$-story houses for one fixed $k \in \mathbb{N}$, we use slightly simpler notation than the one in Algorithm \ref{alg:cap}.
\begin{algorithm}
	\caption{The calculation of the $(k +1)$-variate generating function}\label{alg:cap}
	\begin{algorithmic}
	    \State MaximalConfigurations $=$ all the maximal configurations of length $8k + 1$
	    \State \Comment{this is found by an exhaustive search}
		\State AllowedNodes $=\emptyset$, StartNodes $=\emptyset$, EndNodes $=\emptyset$
		\For{mc $\in$ MaximalConfigurations}
			\State add substring(mc, start $ = 1$, end $ = 4k + 1$) to StartNodes
			\State add substring(mc, start $ = 4k + 1$, end $ = 8k + 1$) to EndNodes
			\For{$1 \le i \le 4k + 1$}
			    \State add substring(mc, start $ = i$, end $ = 4k + i$) to AllowedNodes
			\EndFor
		\EndFor
		\\
		\State $n = $ length(AllowedNodes)
		\State $A = $ null matrix(nrow $ = n$, ncol $ = n$)
		\For {$1 \le i \le n$}
		    \For {$1 \le j \le n$}
		        \If {substring(AllowedNodes$[i]$, start $ = 2$, end $ = 4k + 1$) $ = $ $\quad\backslash\backslash$\\ substring(AllowedNodes$[j]$, start $ = 1$, end $ = 4k$)}
		        \State $r = $ substring(AllowedNodes[j], start $ = 4k + 1$, end $ = 4k + 1$)
		        \If{$r = 0$}
		            \State $A[i, j] = 1$
                \Else
                    \State $A[i, j] = x_r$
                \EndIf
		        \EndIf
		    \EndFor
        \EndFor
		\\
		\State $a = $ null vector(length $ = n$),  $b = $ null vector(length $ = n$)
		\For{$1 \le i \le n$}
		    \If{AllowedNodes$[i] \in $ StartNodes}
		        \State $a[i] = 1$
		        \For{$1 \le r \le k$}
		            \State $c = $ count($r$ in AllowedNodes$[i]$)
		            \State $a[i] = a[i] \cdot x_r^c$
		        \EndFor
		    \EndIf
		    \If{AllowedNodes$[i] \in $ EndNodes}
		        \State $b[i] = 1$
            \EndIf
		\EndFor
		\\
		\State $a_{c_1, c_2, \ldots, c_k, j} = $ the number of maximal configurations of length $j$ with precisely $c_r$ $r$-story houses \Comment{this is found by an exhaustive search for $0 \le j \le 4k$}
		\State $F_1(x_1, x_2, \ldots, x_k, y) = \sum_{j = 0}^{4k} \sum_{c_1, c_2, \ldots, c_k = 0}^{j} a_{c_1, c_2, \ldots, c_k, j} \prod_{r = 1}^{k} x_r^{c_r}y^j$
		\State $F_2(x_1, x_2, \ldots, x_k, y) = \sum_{j = 4k + 1}^{\infty} a^T \cdot A^{j - 4k} \cdot b \cdot y^j = a^T\cdot(I-Ay)^{-1}\cdot A\cdot b\cdot y^{4k+1}$
		\\
		\State\Return $F(x_1, x_2, \ldots, x_k, y) = F_1(x_1, x_2, \ldots, x_k, y) + F_2(x_1, x_2, \ldots, x_k, y)$
	\end{algorithmic}
\end{algorithm}

\subsubsection{Two-story Riviera model}

In the two-story Riviera model, all the houses on our $1 \times n$ tract of land have precisely two stories. This means that, to ensure permissibility of the configuration, each house needs to have at least two empty lots immediately to the east or at least two empty lots immediately to the west of itself (we again assume that there is no obstruction to sunlight on the east and west boundary of the tract of land). Maximal configurations are, as before, those that are permissible, but become impermissible as soon as we add one more house on any of the empty lots. Representing configurations of length $n$ with row vectors $C=(c_1,\dots, c_n) \in \mathbb{N}_0^n$ where $c_i = 2$ denotes that on the lot $i$ ($i \in \{1, 2, \ldots, n\}$) there is a house with $2$ stories (and $c_i = 0$ denotes that the lot $i$ is not occupied by a house), one example of maximal configuration in the two-story Riviera model is $0202000022$. It is easy to check that each of the two stories of every house in this configuration is exposed to sunlight from at least one side (east or west), and that adding another two-story house on any of the empty lots would turn this configuration into an impermissible one. The goal now is to repeat the same procedure as for the original Riviera model and to obtain bivariate generating function which will give us information not only about the number of maximal configurations with a fixed length, but also about the number of maximal configurations with a fixed number of houses. Using Algorithm \ref{alg:cap}, we can construct the digraph that enables us to encode maximal configurations with walks on that digraph.  Analogously as in Subsection \ref{subsec:gen_fun_for_Riviera} we define the matrix function $A(x)$, which is not only the adjacency matrix of the mentioned digraph, but additionally encodes whether a transition from one node of the graph to another one results in the increase of the total number of occupied lots. Due to the dimension of the matrix $A(x)$ and the technicalities in developing bivariate generating function, we omit the details. The precise shape of the generating function $F(x, y)$ (given in Appendix --- see \eqref{it:two-story_Riviera}) was obtained following Algorithm \ref{alg:cap} and using R programming language \cite{Rcite} for creating the necessary objects (matrix $A$, vectors $a$ and $b$) and Maxima software to perform the calculation of the generating function using objects created in R. The first few coefficients in the expansion of $F(x,y)$ are given in Table \ref{tab:gf2}. By inspecting the non-zero coefficients in the table, we see that the ratio $\frac{k}{n}$ is in-between $\frac{2}{7}$ and $\frac{1}{2}$, for large $n$.

\begin{table}[!h]
	\caption{The first few coefficients in the expansion of the bivariate generating function $F(x,y)$.}\label{tab:gf2}
	\small
	\begin{tabular}{c|cccccccccccccccc}
		$k\backslash n$ & $1$ & $y$ & $y^2$ & $y^3$ & $y^4$ & $y^5$ & $y^6$ & $y^7$ & $y^8$ & $y^9$ & $y^{10}$ & $y^{11}$ & $y^{12}$ & $y^{13}$ & $y^{14}$ & $y^{15}$\\\hline
		$1$      & $1$ &     &     &     &     &     &     &     &      &      &      &      &      &      &      &      \\
		$x$      &     & $1$ &     &     &     &     &     &     &      &      &      &      &      &      &      &      \\
		$x^2$    &     &     & $1$ & $3$ & $6$ & $5$ & $3$ & $1$ &      &      &      &      &      &      &      &      \\
		$x^3$    &     &     &     &     &     & $2$ & $4$ & $5$ & $2$  &      &      &      &      &      &      &      \\
		$x^4$    &     &     &     &     &     &     & $1$ & $5$ & $16$ & $27$ & $31$ & $24$ & $13$ & $5$  &  $1$ &      \\
		$x^5$    &     &     &     &     &     &     &     &     &      & $3$  & $12$ & $28$ & $36$ & $29$ &  $14$& $3$  \\
		$x^6$    &     &     &     &     &     &     &     &     &      &      & $1$  & $7$  & $31$ & $80$ & $142$& $177$\\
		$x^7$    &     &     &     &     &     &     &     &     &      &      &      &      &      & $4$  & $24$ & $82$ \\
	\end{tabular}
\end{table}

From Table \ref{tab:gf2} we read that there are $44$ maximal configurations of length $10$, out of which $31$ configurations have precisely $4$ two-story houses built on them. This means that, apart from the maximal configuration that we gave as an example ($0202000022$), there are $43$ more maximal configurations with the same length and $30$ more maximal configurations with the same length and the same number of houses.

The first several values of the integer sequence that counts the number of maximal configurations in the two-story Riviera model of length $n$ can be read as column sums: $1, 1, 3, 6, 7, 8, 11, 18, 30, 44, \dots$ and this sequence is still not a part of the On-Line Encyclopedia of Integer Sequences. On the other hand, by expanding $F(x, y)$ into the formal power series in powers of $x$, (i.e.\ taking row sums) we can see that, beside the mentioned maximal configuration with precisely $4$ two-story houses, there are $122$ more maximal configurations with the same number of houses. The first several values of the integer sequence that counts the number of maximal configurations in the two-story Riviera model with precisely $k$ houses can be read as row sums: $1, 19, 13, 123, 125, 811, 1069, 5435, 8605, 36939, \dots$. This sequence is also not yet included in the OEIS.

By performing the same analysis as in the case of one-story houses, one can
obtain the expected number of occupied lots in a maximal configuration of
length $n$ as $\langle k_2(n) \rangle = 0.388957 \,\,n $, and the 
expected length of a maximal configuration with $k$ two-story houses as
$\langle n_2(k)\rangle = 2.706054 \,\,k$. Since the largest possible number of
buildings in a maximal configuration is, roughly, $n/2$, one can see that the
expected efficiency $\varepsilon _2$ is lower than in the case of one-story
buildings: $\varepsilon _2 = 0.777914$. However, by counting each house as two
basic units (two flats), it follows that on a maximal configuration of the same
length $n$ one can build, on average, $2 \cdot 0.388957 / 0.577203 = 1.34773$
times more basic units by opting for two-story houses. It is reassuring that
our model, although very simple, captures the common wisdom of planners and
builders. It would be interesting to investigate how this gain in 
efficiency is affected by further increase in the number of stories.

\subsubsection{Mixed Riviera model}
By the mixed Riviera model, we refer to the model where one-story houses and two-story houses can be built on the same $1 \times n$ tract of land. Notice that, even though we cannot add additional two-story house on any of the empty lots in the configuration $0202000022$, we could add one-story house without making the configuration impermissible. One possible position to add one-story house to this configuration is $i = 6$. If we set $c_6 = 1$ we get the configuration $0202010022$. This one-story house is not blocking the sun to any of the stories of already built houses, and the one-story house itself is exposed to sunlight. It is easy to check that this new configuration with one-story house is maximal since we cannot add additional stories to any of the empty lots nor to the lot where there is a one-story house. In this model, words that represent configurations have three letters ($0$, $1$ and $2$). However, the transfer matrix method and Algorithm \ref{alg:cap} can again be used to obtain the digraph through which we encode all the maximal configurations in this mixed Riviera model. Once we have the digraph, we have the adjacency matrix, but as before, we use this adjacency matrix for more than just counting the number of maximal configurations of fixed length. Depending on whether a transition from node $i$ to node $j$ results in addition of an empty lot, a lot on which one-story house is built or a lot on which two-story house is built, matrix function $A$ has $1$, $x$ or $y$ on the position $(i, j)$, respectively. Hence, $A$ is a function of two variables ($x$ and $y$) where $x$, rather than $x_1$, is a formal variable associated with the number of one-story houses in the configuration and $y$, rather than $x_2$, is a formal variable associated with the number of two-story houses in the configuration. Denoting the formal variable associated with the length of the configuration by $z$, rather than $y$, and using similar calculations as in Subsection \ref{subsec:gen_fun_for_Riviera} we obtain trivariate generating function. The precise shape of this generating function is again obtained by R and Maxima and it is given in Appendix (see \eqref{it:mixed_Riviera}). By expanding $F(x, y, z)$ into the formal power series in powers of $z$, we get
{\small \begin{align*}
    F & (x, y, z) = 1 \\
    & + yz \\
    & + y^2z^2 \\
    & + 3y^2z^3 \\
    & + (2x + 4)y^2z^4 \\
    & + (2y^3 + (x^2 + 6x + 1)y^2)z^5 \\
    & + (y^4 + 4y^3 + (7x^2 + 6x)y^2)z^6 \\
    & + (5y^4 + (2x^2 + 2x + 3)y^3 + (2x^3 + 11x^2 + 2x)y^2)z^7 \\
    & + ((4x + 13)y^4 + (10x^2 + 4x)y^3 + (x^4 + 8x^3 + 5x^2)y^2)z^8 \\
    & + (3y^5 + (3x^2 + 18x + 15)y^4 + (4x^3 + 18x^2 + 2x)y^3 + (7x^4 + 8x^3 + x^2)y^2)z^9 \\
    & + (y^6 + 12y^5 + (30x^2 + 34x + 8)y^4 + (4x^4 + 16x^3 + 10x^2)y^3 + (2x^5 + 11x^4 + 2x^3)y^2)z^{10} \\
    & + \dots
\end{align*}}%
Our example of a maximal configuration in the mixed Riviera model has length $10$, $1$ one-story house and $4$ two-story houses. From this expansion it is easy to read that there are $33$ more maximal configurations of length $10$ with $1$ one-story house and $4$ two-story houses. Notice that plugging $y = x$ in the trivariate generating function $F(x, y, z)$ gives bivariate generating function in variables $x$ and $z$ where formal variable $x$ is associated with the number of houses (regardless of the number of stories) and formal variable $z$ is still associated with the length of the maximal configuration. Similarly, setting $y = x^2$ we get bivariate generating function where formal variable $x$ is associated with the number of stories on maximal configurations. The most natural sequence related to this model is the one that counts the number of maximal configurations of length $n$. The first several values of this sequence are $1, 1, 3, 6, 10, 18, 27, 45, 79, 130, \dots$. This sequence is not found on the OEIS.

\subsection{AND variant}

As explained at the beginning of this section, the AND variant, unlike the OR variant, requires that each story of each house gets sunlight from the east and from the west, and in the one-story case this comes down to the classical Flory polymer model introduced by Flory \cite{Flory} already in 1939. The sequence $(a_n)$ that counts the number of maximal configurations of length $n$ in the one-story Flory model is the famous Padovan sequence (see \href{https://oeis.org/A000931}{A000931}) with offset $6$ or, equivalently, the number of compositions (ordered partitions) of number $n+3$ into parts $2$ and $3$ (i.e.\ \href{https://oeis.org/A182097}{A182097} with offset $3$). Moreover, if we are interested in the asymptotic formula for $a_n$, we can obtain it as explained in Remark \ref{rem:PF-Riviera}. It holds that $a_n \sim C\,\lambda^n$ as $n$ tends to infinity where $\lambda$ is the well-known plastic number $\lambda = \frac{1}{w} = 1.324718$. The numerical value of $C=\frac{\lambda^3+\lambda^2+\lambda}{2\lambda+3}=0.956611$ can, as before, determined using Theorem \ref{tm:asymp}. In the next two subsections we discuss variants of Flory model where we can have houses that have more than one story.

\begin{remark}
    Note that there are more maximal configurations in the Riviera model than in the Flory model since the Perron-Frobenius eigenvalue in the Riviera model is bigger than the plastic number (which is the Perron-Frobenius eigenvalue in the Flory model). It is also interesting to compare these constants with $2$ since there are $2^n$ binary sequences of length $n$ when we do not impose any restrictions on those sequences.
\end{remark}

\subsubsection{Multi-story Flory models}
Using the same technique as in the last two subsections, we could easily obtain generating functions related to the multi-story Flory models, but sequences that count the number of maximal configurations (in multi-story Flory models) with fixed length or with fixed number of houses are some well-known sequences. Therefore, we will just relate our sequences to those already known by establishing the appropriate bijections.

The simplest way to explain the connection of the sequences that arise from the multi-story Flory models and some already known sequences is by using an example. Let us consider the two-story Flory model. In this model, all the houses have precisely two stories and each story of each house needs to get sunlight from both east and west (we assume that there is no obstruction to sunlight on the east and west boundary of the tract of land). This implies that each house needs to have at least two empty lots immediately to the east and immediately to the west of it (except those houses that are near the boundary). Let us have a look at all the maximal configurations of length $6$. Those are $200200, 200020, 020020, 200002, 020002$ and $002002$. Clearly, between each two houses we can have $2, 3$ or $4$ empty lots. Less than $2$ empty lots would mean that those houses are not getting the sunlight from both east and west (hence, the permissibility would be violated) and more than $4$ empty lots would mean that we can add additional house between those two houses without violating permissibility (hence, the maximality would be violated). This implies that if we take a block of lots that includes all the empty lots to the west of the house and the house itself, this block can have length $3, 4$ or $5$. This is true for all the houses except maybe the first one since this one can get the sunlight from outside of the tract of land. For the first house this block including all the empty lots to the west of it and the first house itself can be of size $1$, $2$ or $3$. For that reason, we artificially add $2$ empty lots in front of our maximal configurations. Also, if the configuration doesn't end with a house, empty lots at the end of the configuration will not be a part of any block. To solve this problem, we artificially add $002$ at the end of each maximal configuration since in this way we will create a block at the end that has length $3$, $4$ or $5$. After adding these $5$ additional lots, we ended up with strings of length $11$. Each of these strings is different and they can be divided into blocks that contain one house and all the empty lots preceding that house. As explained, these blocks will be of length $3$, $4$ and $5$ and hence will give as a representation of number $11$ as a sum of numbers $3$, $4$ and $5$. In our example we have
\begin{align*}
    200200 & \to \textcolor{red}{00}|200200|\textcolor{red}{002} \to \underbrace{002}_3\underbrace{002}_3\underbrace{00002}_5 \\
    200020 & \to \textcolor{red}{00}|200020|\textcolor{red}{002} \to \underbrace{002}_3\underbrace{0002}_4\underbrace{0002}_4 \\
    020020 & \to \textcolor{red}{00}|020020|\textcolor{red}{002} \to \underbrace{0002}_4\underbrace{002}_3\underbrace{0002}_4 \\
    200002 & \to \textcolor{red}{00}|200002|\textcolor{red}{002} \to \underbrace{002}_3\underbrace{00002}_5\underbrace{002}_3 \\
    020002 & \to \textcolor{red}{00}|020002|\textcolor{red}{002} \to \underbrace{0002}_4\underbrace{0002}_4\underbrace{002}_3 \\
    002002 & \to \textcolor{red}{00}|002002|\textcolor{red}{002} \to \underbrace{00002}_5\underbrace{002}_3\underbrace{002}_3
\end{align*}
It is also trivial to go in the other direction. If we are given one possible ordered partition of number $11$ into parts $3$, $4$ and $5$, we can just concatenate the blocks whose lengths will correspond to those parts and in the end we just remove the first two and the last three lots. For example
\begin{equation*}
    11 = 3 + 5 + 3 \to \underbrace{002}_{3}\underbrace{00002}_{5}\underbrace{002}_{3} \to \textcolor{red}{00}|200002|\textcolor{red}{002} \to 200002.
\end{equation*}
It is clear from here that there is a bijection between maximal configurations of the two-story Flory model of length $n$ and compositions (ordered partitions) of number $n + 5$ into parts $3$, $4$ and $5$. The sequence that counts the number of ordered partitions of $n$ into parts $3$, $4$ and $5$ can be found on the OEIS under name \href{https://oeis.org/A017818}{A017818}.

Of course, there is nothing special about the two-story Flory model and the same reasoning can be applied to any multi-story Flory model. If we consider $k$-story Flory model, then we get maximal configurations with at least $k$ empty lots and at most $2k$ empty lots between each two occupied lots. Adding $k$ empty lots in front of each maximal configuration and $k$ empty lots and one occupied lot at the end of each configuration we get a string of length $n + 2k + 1$. Splitting this string into blocks containing an occupied lot and all the empty lots to the west of it, we get decomposition of the number $n + 2k + 1$ into parts $k + 1, k + 2, \ldots, 2k + 1$. On the other hand, if we start from a composition of the number $n + 2k + 1$ into parts $k + 1, k + 2, \ldots, 2k + 1$, we can trivially reconstruct the corresponding maximal configuration. Hence, we have the following proposition. 
\begin{proposition}
    The number of maximal configurations with fixed length $n \in \mathbb{N}$ in the $k$-story Flory model ($k \in \mathbb{N}$) is equal to the number of compositions (ordered partitions) of number $n + 2k + 1$ into parts $k + 1, k + 2, \ldots, 2k + 1$.
\end{proposition}
 For $k \in \{3, \ldots, 9\}$ those sequences can be found on the OEIS under the following names
\begin{itemize}
    \item $k = 3$ -- compositions of $n$ into parts $p$ where $4 \le p \le 7$ -- \href{https://oeis.org/A017829}{A017829},
    \item $k = 4$ -- compositions of $n$ into parts $p$ where $5 \le p \le 9$ -- \href{https://oeis.org/A017840}{A017840},
    \item $k = 5$ -- compositions of $n$ into parts $p$ where $6 \le p \le 11$ -- \href{https://oeis.org/A017851}{A017851},
    \item $k = 6$ -- compositions of $n$ into parts $p$ where $7 \le p \le 13$ -- \href{https://oeis.org/A017862}{A017862},
    \item $k = 7$ -- compositions of $n$ into parts $p$ where $8 \le p \le 15$ -- \href{https://oeis.org/A017873}{A017873},
    \item $k = 8$ -- compositions of $n$ into parts $p$ where $9 \le p \le 17$ -- \href{https://oeis.org/A017884}{A017884},
    \item $k = 9$ -- compositions of $n$ into parts $p$ where $10 \le p \le 19$ -- \href{https://oeis.org/A017895}{A017895}.
\end{itemize}
Since maximal configurations of length $n = 1$ are in bijection with the ordered partitions of the number $2k + 2$, we have to look at all the above sequences with the offset of $2k + 1$.

Note that it is straightforward to obtain the recurrence relation for these sequences (and from that recurrence relation also the generating function) regardless of the value of $k$. Hence, even though the sequences for $k \ge 10$ do not appear on the OEIS, we can easily calculate their elements. Let us explain how to get the recurrence relation and the generating function in the case $k = 3$ and then formulate the general result. Denote by $a_n$ the number of compositions of $n$ into parts $4, 5, 6$ and $7$. Clearly, $a_0=1$ since there is one way to get $0$ from parts $4, 5, 6$ and $7$, we do not take any of the parts. For integers smaller than the smallest available part there are zero compositions, hence $a_1 = a_2 = a_3 = 0$. Since available parts are integers between $k + 1$ and $2k + 1$, they can be composed in only one way, therefore $a_4 = a_5 = a_6 = a_7 = 1$. For any $n \ge 8$ the logic is that we can first compose $n - 4$ from available parts and then just add $4$ and similar for other available parts. Hence, we have the following recurrence relation
\begin{equation*}
    a_n = a_{n - 4} + a_{n - 5} + a_{n - 6} + a_{n - 7}.
\end{equation*}
From this recurrence relation, we clearly have that the generating function of the sequence $a_n$ is given with
\begin{equation*}
    F(x) = \frac{1}{1 - x^4 - x^5 - x^6 - x^7}.
\end{equation*}
It is clear now that we can follow the same logic in the general case with $k$-story houses.
\begin{proposition}
    Fix $k \in \mathbb{N}$. Let $(a_n)_{n \ge 0}$ be the integer sequence where
    \begin{equation*}
        a_n = \textnormal{the number of compositions of $n$ into parts $k + 1, k + 2, \ldots, 2k + 1$}.
    \end{equation*}
    Then for every $n \ge 2k + 2$ the sequence $(a_n)_{n \ge 0}$ satisfies the recurrence relation
    \begin{equation*}
        a_n = \sum_{i = k + 1}^{2k + 1} a_{n - i},
    \end{equation*}
    with initial conditions
    \begin{align*}
        a_0 & = 1, \\
        a_i & = 0, \quad 1 \le i \le k \\
        a_i & = 1, \quad k + 1 \le i \le 2k+1.
    \end{align*}
    Furthermore, the generating function $F(x)$ of the sequence $(a_n)_{n \ge 0}$ is given with
    \begin{equation*}
        F(x) = \frac{1}{1 - \sum_{i = k + 1}^{2k + 1} x^i}.
    \end{equation*}
\end{proposition}

The situation with maximal configurations with fixed number of houses in the $k$-story Flory model is even simpler. Let us again describe what happens in the case $k = 2$ and then formulate the general result. To find all the maximal configurations with exactly $3$ houses, we start with the maximal configuration $2002002$ which has $3$ houses. Now we can obtain all the maximal configurations with exactly $3$ houses from this one in the following way: for each house, we choose whether we want to put $0$, $1$ or $2$ additional empty lots to the west of it and for the last house, we also choose whether we want to put $0$, $1$ or $2$ additional empty lots to the east of it. Since we have $3$ houses, we need to make $4$ decisions and we have $3$ options in each of those $4$ decisions. Hence, to each ordered quadruple of elements $0$, $1$ and $2$, we can assign a unique maximal configuration with precisely $3$ houses. For example
\begin{align*}
    (1, 2, 2, 1) & \to \textcolor{red}{0}|200|\textcolor{red}{00}|200|\textcolor{red}{00}|2|\textcolor{red}{0}, \\
    (0, 1, 2, 1) & \to 200|\textcolor{red}{0}|200|\textcolor{red}{00}|2|\textcolor{red}{0}, \\
    (1, 1, 2, 0) & \to \textcolor{red}{0}|200|\textcolor{red}{0}|200|\textcolor{red}{00}|2, \\
    (0, 0, 0, 0) & \to 2002002.
\end{align*}
Since we start with $2$ empty lots between each two houses, by adding $0$, $1$ or $2$ more empty lots, we end up with $2$, $3$ or $4$ empty lots between each two houses which still guarantees maximality and we obviously cannot violate permissibility with adding empty lots. Also, since we started from a configuration that has houses on the easternmost and on the westernmost lots, adding $0$, $1$ or $2$ additional empty lots before the first house or after the last house, will not violate maximality. On the other hand, if we are given a maximal configuration in the two-story Flory model with precisely $3$ houses, we can easily obtain the corresponding ordered quadruple of elements $0$, $1$ and $2$. Each maximal configuration has to have $0$, $1$ or $2$ empty lots before the first house and after the last one and it has to have at least $2$ and at most $4$ empty lots between each two houses. First element of the ordered quadruple corresponds to the number of empty lots in front of the first house. Counting the number of empty lots between each two houses and subtracting $2$ gives as all the other elements of our ordered quadruple, except for the last one. The last one corresponds to the number of empty lots after the last house. For example
\begin{equation*}
    02000200200 \to \underbrace{0}_{1}2\underbrace{000}_{3}2\underbrace{00}_{2}2\underbrace{00}_{2} \to (1, 3-2, 2-2, 2) = (1, 1, 0, 2).
\end{equation*}
Hence, the total number of maximal configurations with precisely $3$ houses in the two-story Flory model is $3^4$. Again, there is nothing special about the two-story Flory model and the same logic can be applied to any $k$-story Flory model. If we want to count the number of maximal configurations with precisely $n$ houses in the $k$-story Flory model, we start with the configuration
\begin{equation*}
    \overbrace{k\underbrace{00\ldots00}_{k\textnormal{ zeros}}k\underbrace{00\ldots00}_{k\textnormal{ zeros}}k0\ldots0k}^{n \textnormal{ houses}}
\end{equation*}
and then we choose ordered $(n+1)$-tuple of elements $0, 1, \ldots, k$. Denote this $(n+1)$-tuple with $(k_0, k_1, \ldots, k_n)$. To this $(n+1)$-tuple we assign the maximal configuration
\begin{equation*}
    \overbrace{\underbrace{00\ldots00}_{k_0 \textnormal{ zeros}}k\underbrace{00\ldots00}_{k\textnormal{ zeros}}\underbrace{00\ldots00}_{k_1 \textnormal{ zeros}}k\underbrace{00\ldots00}_{k\textnormal{ zeros}}\underbrace{00\ldots00}_{k_2 \textnormal{ zeros}}k0\ldots0\underbrace{00\ldots00}_{k_{n - 1} \textnormal{ zeros}}k\underbrace{00\ldots00}_{k_n \textnormal{ zeros}}}^{n \textnormal{ houses}}
\end{equation*}
In this way we obtain a configuration where in front of the first house and after the last house we have at most $k$ empty lots, and between each two houses we have between $k$ and $2k$ empty lots. Such configurations are obviously maximal. On the other hand, if we start from a maximal configuration with precisely $n$ houses, we can trivially reconstruct the corresponding $(n+1)$-tuple of elements $0, 1, \ldots, k$.

\begin{proposition}
    The number of maximal configurations with fixed number of houses $n \in \mathbb{N}$ in the $k$-story Flory model ($k \in \mathbb{N}$) is equal to the the number of ordered $(n+1)$-tuples of elements $0, 1, \ldots, k$ which is equal to $(k + 1)^{n + 1}$.
\end{proposition}

\subsubsection{Mixed Flory model}
Just as in the case of the mixed Riviera model, mixed Flory model refers to the model where one-story and two-story houses can be built on the same $1 \times n$ tract of land. For this model we again obtain trivariate generating function $F(x, y, z)$, where $x$ is a formal variable associated with the number of one-story houses in the configuration, $y$ is a formal variable associated with the number of two-story houses in the configuration and $z$ is a formal variable associated with the length of the configuration. The explicit form of this trivariate generating function was obtained using R and Maxima and it is given in Appendix (see \eqref{it:mixed_flory}). By expanding $F(x ,y ,z)$ into formal power series in powers of $z$, we get
\begin{align*}
    F(x, y, z) & = 1 \\
    & + yz \\
    & + 2yz^2 \\
    & + (3y + x^2)z^3 \\
    & + (y^2 + 2y + 2x^2)z^4 \\
    & + (3y^2 + y + x^3 + x^2)z^5 \\
    & + (6y^2 + 2x^2y + 2x^3)z^6 \\
    & + (y^3 + 7y^2 + 6x^2y + x^4 + x^3)z^7 \\
    & + (4y^3 + 6y^2 + (2x^3 + 8x^2)y + 3x^4)z^8 \\
    & + (10y^3 + (3x^2 + 3)y^2 + (6x^3 + 6x^2)y + x^5 + 3x^4)z^9 \\
    & + (y^4 + 16y^3 + (12x^2 + 1)y^2 + (2x^4 + 8x^3 + 2x^2)y + 4x^5 + x^4)z^{10} + \dots
\end{align*}
The first several elements of the sequence that counts the number of different maximal configurations of length $n$ are $1, 2, 4, 5, 6, 10, 16, 23, 32, 47, \dots$ This sequence does not appear on the OEIS.

\section{One-story model on an \texorpdfstring{$m\times n$}{mxn} grid for \texorpdfstring{$m=2,3$}{m=2,3} and \texorpdfstring{$n\in\bbN$}{n∈N}}\label{sec:mxnForSmallm}

In this section we come back to the original settlement model introduced in \cite{PSZ-21} in which we have a rectangular $m \times n$ tract of land where $m > 1$ and the sun can come from the east, south and west. The question that we were interested in is what is the total number of all maximal configurations for all grid sizes $m \times n$. This is exactly the question posed in \cite[Question 2.2]{PSZ-21-2}. We give a partial answer to this question when the considered tract of land is narrow. As already announced in the introduction, we were unable to obtain a closed-form formula for all grid sizes, but for grids of size $2\times n$ and $3\times n$ we derive the (bivariate) generating functions counting the number of maximal configurations. The method we use is again the transfer matrix method which could be similarly adopted for bigger $m$'s, but the calculations soon get increasingly infeasible. In this section we only consider one-story houses and we do not look at the Flory's version of the problem where each house would need to get sunlight from east, south and west. Of course, those models could be studied with the same method.

Recall that every configuration on an $m \times n$ tract of land can be encoded as a $m \times n$ matrix $C$, where $c_{i,j} = 1$ if the lot $(i, j)$ is occupied and $c_{i,j} = 0$ otherwise. We now treat all the possible different columns of such matrices as letters in our alphabet. Notice that those are just binary numbers of length $m$ written as column vectors. Once we interpret each column of such a matrix as a letter, every configuration can be viewed as a word of length $n$. As already explained in the introduction, the transfer matrix method can be applied if the properties of permissibility and maximality of a configuration can be verified by inspecting only finite size patches of a given configuration. Luckily, in our model, the properties of maximality and permissibility are locally verifiable, i.e.\ if one wants to check whether a state of a certain lot (occupied or unoccupied) has caused the configuration to be impermissible or not maximal, one only needs to check the situation on the lots in a finite radius of the observed lot. These verifications whether some configuration is permissible and maximal are done directly on the matrix $C$, but as soon as we find a forbidden pattern, we forbid the corresponding word. More precisely, to verify that a configuration is permissible, one needs to check that no occupied lot $(i, j)$ borders simultaneously with tree other occupied lots to its east, south and west. This can be done by inspecting whether the forbidden pattern shown in Figure \ref{fig:forbidden_pattern} appears anywhere in the matrix $C$.
\begin{figure}
	\begin{tikzpicture}[scale = 0.8]
	\filldraw [fill=white, draw=black] (0,0) rectangle (1,1);
	\filldraw [fill=white, draw=black] (1,0) rectangle (2,1);
	\filldraw [fill=white, draw=black] (2,0) rectangle (3,1);
	\filldraw [fill=white, draw=black] (1,-1) rectangle (2,0);
	\node[] at (0.5,0.5) {$1$};
	\node[] at (1.5,0.5) {$1$};
	\node[] at (2.5,0.5) {$1$};
	\node[] at (1.5,-0.5) {$1$};
	\end{tikzpicture}
	\caption{The forbidden pattern.}\label{fig:forbidden_pattern}
\end{figure}
To verify that a configuration is maximal, one needs to additionally check that no unoccupied lots can be built on. Since it holds that
\begin{equation*}
	c_{i, j} =0 \Longrightarrow  \begin{array}{cc}
	(c_{i, j - 2} = 1 \mbox{ and } c_{i, j - 1} = 1 \mbox{ and } c_{i + 1, j - 1} = 1) \mbox{ or } \\
	(c_{i, j + 1} = 1 \mbox{ and } c_{i, j + 2} = 1 \mbox{ and } c_{i + 1, j + 1} = 1) \mbox{ or } \\
	(c_{i - 1, j - 1} = 1 \mbox{ and } c_{i - 1, j} = 1 \mbox{ and } c_{i - 1, j + 1} = 1) \mbox{ or } \\
	(c_{i, j - 1} = 1 \mbox{ and } c_{i, j + 1} = 1 \mbox{ and } c_{i + 1, j} = 1),
	\end{array}
\end{equation*}
this can be done by inspecting all the lots surrounding the lot $(i, j)$ that appear in the above expression. Those are precisely the shaded lots in Figure \ref{fig:maximality_check}.

\begin{figure}
	\begin{tikzpicture}[scale = 0.8]
	\filldraw [fill=blue!40!white, draw=black] (0,0) rectangle (1,1);
	\filldraw [fill=blue!40!white, draw=black] (1,0) rectangle (2,1);
	\filldraw [fill=white, draw=black] (2,0) rectangle (3,1);
	\filldraw [fill=blue!40!white, draw=black] (3,0) rectangle (4,1);
	\filldraw [fill=blue!40!white, draw=black] (4,0) rectangle (5,1);
	\filldraw [fill=blue!40!white, draw=black] (1,-1) rectangle (2,0);
	\filldraw [fill=blue!40!white, draw=black] (2,-1) rectangle (3,0);
	\filldraw [fill=blue!40!white, draw=black] (3,-1) rectangle (4,0);
	\filldraw [fill=blue!40!white, draw=black] (1,1) rectangle (2,2);
	\filldraw [fill=blue!40!white, draw=black] (2,1) rectangle (3,2);
	\filldraw [fill=blue!40!white, draw=black] (3,1) rectangle (4,2);
	\node[] at (2.5,0.5) {\small $i, j$};
	\end{tikzpicture}
	\caption{Lots that need to be checked to verify maximality of a configuration.}\label{fig:maximality_check}
\end{figure}

Therefore, the transfer matrix method can be applied in a completely analogous way as before, but as the number of rows grows, the alphabet becomes bigger so we only consider the cases $m = 2$ and $m = 3$.

\subsection{\texorpdfstring{$2 \times n$}{2xn}}

In this subsection, we derive the bivariate generating function from which we can extract the sequence that counts the total number of different maximal configurations of fixed length $n$ (and width $2$) and the sequence that counts the total number of different maximal configurations with precisely $k$ houses on all the $2 \times n$ grids ($n \in \mathbb{N}$).
One example of a maximal configuration on a $2 \times 6$ tract of land is
\begin{equation*}
    \begin{bmatrix}
        1 & 1 & 0 & 1 & 1 & 1 \\
        1 & 1 & 1 & 1 & 0 & 1
    \end{bmatrix}.
\end{equation*}
Since we can apply the transfer matrix method, we can again use Algorithm \ref{alg:cap} to construct the digraph with which we can encode all the maximal configurations. Once we have the digraph, we can define the matrix function $A(x)$. This matrix is obtained from the adjacency matrix of the corresponding digraph, but instead of putting $1$ on the position $(i, j)$ when the transition from node $i$ to node $j$ is possible, we put $x$ if the transition from node $i$ to node $j$ adds one house to the configuration and we put $x^2$ if the transition from node $i$ to node $j$ adds two houses to the configuration (notice that it is impossible that the transition from node $i$ to node $j$ adds zero houses to the configuration because the word $(0, 0)^T$ never appears since it would cause that the resulting configuration is not maximal). As before, the bivariate generating function is obtained using R and Maxima and the precise shape can be found in the Appendix (see \eqref{it:2xn}). The first few coefficients in the expansion of $F(x, y)$ are given in Table \ref{tab:gf3}. By inspecting the non-zero coefficients in the table, we see that the ratio $\frac{k}{n}$ is in-between $1$ and $\frac{5}{3}$, for large $n$.

\begin{table}[!h]
	\caption{The first few coefficients in the expansion of the bivariate generating function $F(x,y)$.}\label{tab:gf3}
	\small
	\begin{tabular}{c|cccccccccccccccccc}
		$n\backslash k$ & $1$ & $x$ & $x^2$ & $x^3$ & $x^4$ & $x^5$ & $x^6$ & $x^7$ & $x^8$ & $x^9$ & $x^{10}$ & $x^{11}$ & $x^{12}$ & $x^{13}$ & $x^{14}$ & $x^{15}$ & $x^{16}$ & $x^{17}$\\\hline
		$1$      & $1$ &     &     &     &     &     &     &     &      &      &      &      &      &      &      &      &      &      \\
		$y$      &     &     & $1$ &     &     &     &     &     &      &      &      &      &      &      &      &      &      &      \\
		$y^2$    &     &     &     &     & $1$ &     &     &     &      &      &      &      &      &      &      &      &      &      \\
		$y^3$    &     &     &     &     &     & $4$ &     &     &      &      &      &      &      &      &      &      &      &      \\
		$y^4$    &     &     &     &     &     &     & $4$ & $2$ &      &      &      &      &      &      &      &      &      &      \\
		$y^5$    &     &     &     &     &     &     &     & $4$ & $5$  & $1$  &      &      &      &      &      &      &      &      \\
		$y^6$    &     &     &     &     &     &     &     &     & $4$  & $4$  & $8$  &      &      &      &      &      &      &      \\
		$y^7$    &     &     &     &     &     &     &     &     &      & $4$  & $4$  & $18$ & $3$  &      &      &      &      &      \\
		$y^8$    &     &     &     &     &     &     &     &     &      &      & $4$  & $4$  & $25$ & $16$ & $1$  &      &      &      \\
		$y^9$    &     &     &     &     &     &     &     &     &      &      &      & $4$  & $4$  & $33$ & $31$ & $13$ &      &      \\
		$y^{10}$ &     &     &     &     &     &     &     &     &      &      &      &      & $4$  & $4$  & $41$ & $42$ & $50$ &  $4$ \\
	\end{tabular}
\end{table}

Notice that there are $16$ maximal configurations on $2 \times 6$ tract of land and $8$ of those have $10$ houses. This means that in addition to our example, there are $15$ more maximal configurations on $2 \times 6$ tract of land and $7$ more on that tract of land that have precisely $10$ houses. The first several elements of the integer sequence that counts the number of maximal configurations of fixed length $n$ (and width $2$) can be read as row sums: $1, 1, 4, 6, 10, 16, 29, 50, 85, 145, \dots$ and this sequence is still not included in the OEIS. On the other hand, by expanding $F(x, y)$ into the formal power series in powers of $x$ (i.e.\ taking column sums), we can see that, beside the mentioned maximal configuration with precisely $10$ houses, there are $15$ more maximal configurations with the same number of houses. The first several elements of the integer sequence that counts the number of maximal configurations with precisely $k$ houses on all the $2 \times n$ grids ($n \in \mathbb{N}$) can be read as column sums: $0, 1, 0, 1, 4, 4, 6, 9, 9, 16, \dots$ This sequence is not found on the OEIS.

Here again we compute the expected number of occupied lots in a maximal 
$2 \times n$ configuration and the expected length $n$ of a maximal $2 \times n$
configuration with $k$ houses. The obtained values are 
$$\langle k_{2,n}(n) \rangle = 1.437496 \,\,n \quad {\rm and} \quad
\langle n_{2,n}(k) \rangle = 0.724696 \,\,k.$$

\subsection{\texorpdfstring{$3 \times n$}{3xn}}

In this subsection, we derive the bivariate generating function from which we can extract the sequence that counts the total number of different maximal configurations of fixed length $n$ (and width $3$) and the sequence that counts the total number of different maximal configurations with precisely $k$ houses on all the $3 \times n$ grids ($n \in \mathbb{N}$).
One example of a maximal configuration on a $3 \times 5$ tract of land is
\begin{equation*}
    \begin{bmatrix}
        0 & 1 & 1 & 0 & 1 \\
        1 & 1 & 1 & 1 & 1 \\
        1 & 0 & 0 & 0 & 1
    \end{bmatrix}.
\end{equation*}
Similarly as in the case $2\times n$ we obtain the bivariate generating function (see \eqref{it:3xn}). The first few coefficients in the expansion of $F(x, y)$ are given in Table \ref{tab:gf4}. By inspecting the non-zero coefficients in the table, we see that the ratio $\frac{k}{n}$ is in-between $\frac{3}{2}$ and $\frac{7}{3}$, for large $n$.

\begin{table}[!h]
	\caption{The first few coefficients in the expansion of the bivariate generating function $F(x,y)$.}\label{tab:gf4}
	\small
	\begin{tabular}{c|cccccccccccccccccc}
		$n\backslash k$ & $1$ & $x$ & $x^2$ & $x^3$ & $x^4$ & $x^5$ & $x^6$ & $x^7$ & $x^8$ & $x^9$ & $x^{10}$ & $x^{11}$ & $x^{12}$ & $x^{13}$ & $x^{14}$ & $x^{15}$ & $x^{16}$ & $x^{17}$\\\hline
		$1$      & $1$ &     &     &     &     &     &     &     &      &      &      &      &      &      &      &      &      &      \\
		$y$      &     &     &     & $1$ &     &     &     &     &      &      &      &      &      &      &      &      &      &      \\
		$y^2$    &     &     &     &     &     &     & $1$ &     &      &      &      &      &      &      &      &      &      &      \\
		$y^3$    &     &     &     &     &     &     &     & $9$ & $1$  &      &      &      &      &      &      &      &      &      \\
		$y^4$    &     &     &     &     &     &     &     &     & $4$  & $8$  & $7$  &      &      &      &      &      &      &      \\
		$y^5$    &     &     &     &     &     &     &     &     &      &      & $12$ & $8$  & $20$ & $1$  &      &      &      &      \\
		$y^6$    &     &     &     &     &     &     &     &     &      &      &      &      & $24$ & $12$ & $65$ & $4$  &      &      \\
		$y^7$    &     &     &     &     &     &     &     &     &      &      &      &      &      & $24$ & $12$ & $84$ & $122$& $27$ \\
		$y^8$    &     &     &     &     &     &     &     &     &      &      &      &      &      &      & $4$  & $40$ & $40$ &$228$ \\
		$y^9$    &     &     &     &     &     &     &     &     &      &      &      &      &      &      &      &      & $24$ & $44$ \\
	\end{tabular}
\end{table}

Notice that there are $41$ maximal configurations on $3 \times 5$ tract of land and $12$ of those have $10$ houses. This means that in addition to our example, there are $40$ more maximal configurations on $3 \times 5$ tract of land and $11$ more on that tract of land that have precisely $10$ houses. The first several elements of the integer sequence that counts the number of maximal configurations of fixed length $n$ (and width $3$) can be read as row sums: $1, 1, 10, 19, 41, 105, 269, 651, 1560, 3861, \dots$ and this sequence is still not included in the OEIS. On the other hand, by expanding $F(x, y)$ into the formal power series in powers of $x$ (i.e.\ taking column sums), we can see that beside the mentioned maximal configuration with precisely $10$ houses, there are $18$ more maximal configurations with the same number of houses. The first several elements of the integer sequence that counts the number of maximal configurations with precisely $k$ houses on all the $3 \times n$ grids ($n \in \mathbb{N}$) can be read as column sums: $0, 0, 1, 0, 0, 1, 9, 5, 8, 19$. This sequence cannot be found on the OEIS.


In this case, the expected number of occupied lots in a maximal 
$3 \times n$ configuration and the expected length $n$ of a maximal $3 \times n$
configuration with $k$ houses are given as
$$\langle k_{3,n}(n) \rangle = 2.071886 \,\,n \quad {\rm and} \quad
\langle n_{3,n}(k) \rangle = 0.503345 \,\,k.$$
Since the largest possible number of occupied lots is, roughly,
of the order $5n/3$ in the $2 \times n$ strip and $7n/3$ in the $3 \times n$
strip for large $n$, the respective efficiencies are given by 
$$ \varepsilon _{2 \times n} =  0.862498 \quad {\rm and } \quad
\varepsilon _{3 \times n} = 0.887951.$$
One can see that the efficiency of $2 \times n$ strip is lower than for the
Riviera model. This low value is easily explained by the effects
of the fully built lower row. As expected, its effects decrease with the
increasing strip width.

\section{Concluding remarks}\label{sec:concluding}

In this paper we have considered a one-dimensional toy-model of a settlement
planning problem introduced recently by three of the present authors.
In particular, we studied maximal configurations of buildings in a narrow
strip of land oriented east--west subject to the condition that each
building must receive sunlight from either east or west or from both sides.
We have formulated the problem of enumerating such maximal configurations
as a problem of counting binary words of a given
length satisfying certain additional conditions on the allowed patterns. 
By reducing the new problem to counting certain types of walks on a small
directed graph with only six vertices we were able to employ the transfer
matrix method which
yielded generating functions for the enumerating sequences, and hence also
their asymptotics. Along the way we discovered that our maximal configurations
are equinumerous with certain type of restricted permutations and provided
a combinatorial proof of this fact by explicitly constructing a bijection
between two sets.

We have also considered some generalizations of the original problem, such
as allowing the buildings to have more than one floor and varying the type
of restrictions on the sunlight direction. The methods developed
on the toy model were successfully adapted to the more complex settings and
allowed us to obtain multivariate generating functions for the corresponding
enumerating sequences. We have also obtained generating functions for the
sequences enumerating maximal configurations on strips of width 2 and 3,
in most cases obtaining sequences not (yet) in the OEIS.

There are other ways, not explored here, to formulate the original problem
and hence to extend our results. For example, one could consider our
problem ``in negative'' and consider unoccupied places instead of occupied
ones. In that case, the unoccupied places in a maximal configuration must
form a dominating set. For our Riviera model, the maximality condition implies
that such a dominating set would induce a graph of maximum degree one, hence
a dissociation. (Dissociations interpolate between matchings and independent
sets; see \cite{Bock1,Bock2} for definition and some basic properties.)
We are not aware of any results on such dominating sets. 
However, closely related (and less general) independent dominating sets
are being intensely studied, along with other types of domination in graphs.

Our problem could be also formulated and studied on other types of lattices.
Some of them, say the hexagonal one, could be better suited to modeling
real planning applications. On the other hand, the triangular lattice 
might prove more tractable and might lead to closed-form results for the
considered quantities.

Another problem, closely related to the present one, is to study temporal
evolution of built configurations subject to given rules. It is known that
the jamming density of the static and dynamic variant of the Flory model
is not the same; in the static cases, all configurations are considered to
be equally likely, while in the dynamic case some of them are less likely 
to evolve than some others. It would be interesting to explore how additional
restrictions (with respect to the Flory model) affect the difference.

Finally, it would be interesting to model the evolution of configurations
in terms of antagonistic games. It seems plausible that the interest of
a developer is not aligned with the interests of inhabitants -- one would
prefer more buildings, hence more profit, while the others might prefer
more sunshine, hence less profit for the former. Some toy-model simulations
of several variants on dynamics are currently under way.

\section*{Acknowledgments} 
\noindent
We want to express our deep gratitude to Pavel Krapivsky from Boston
University for contacting us and starting an extremely
interesting and fruitful communication that is still ongoing and through
which we learned so much from him as an expert in the field. He was the
one who inspired us to write this particular paper and to consider many
more interesting questions that we are still working on.
Partial support of the Slovenian ARRS (Grant no. J1-3002)
is gratefully acknowledged by T. Do\v{s}li\'c.

\bibliographystyle{bababbrv-fl}
\bibliography{literature}

\appendix

\section{}

Here we give explicit expressions for generating functions related to models discussed in Section \ref{sec:multi-story} and Section \ref{sec:mxnForSmallm}.
\newline
\begin{itemize}
	\item  Bivariate generating function for the two-story Riviera model is given by
	{\tiny \begin{align}\label{it:two-story_Riviera}
		F(x, y) & = -\frac{p(x, y)}{q(x, y)}, \\\nonumber
		p(x, y) & = x^3y^9 + 3x^3y^8 + 3x^3y^7 + (x^3 - x^2)y^6 - 3x^2y^5 - 4x^2y^4 + (x - 3x^2)y^3 - x^2y^2 - xy - 1, \\\nonumber
		q(x, y) & = x^3y^{10} + 2x^3y^9 + x^3y^8 - x^2y^7 - 2x^2y^6 - 2x^2y^5 - x^2y^4 - xy^3 + 1,
		\end{align}}%
	where $x$ is a formal variable associated with the number of two-story houses in the configuration and $y$ is a formal variable associated with the length of the configuration.\\
	\item Trivariate generating function for the mixed Riviera model is given by
	{\tiny \begin{align} \label{it:mixed_Riviera}
		F(x, y, z) & = \frac{p(x, y, z)}{q(x, y, z)}, \\\nonumber
		p(x, y, z) & = x^6y^3z^{17} + ((x^4 - x^5)y^3 + x^6y^2)z^{15} + ((2x^4 - x^5)y^3 - x^6y^2)z^{14} \\\nonumber
		& \qquad + ((-3x^4 - x^3)y^3 - x^5y^2)z^{13} + ((-2x^4 - 3x^3 + 2x^2)y^3 + (x^4 - 2x^5)y^2)z^{12} \\\nonumber
		& \qquad + ((-3x^3 - x^2)y^3 - 4x^4y^2)z^{11} + ((-x^3 - 8x^2)y^3 + x^3y^2)z^{10} \\\nonumber
		& \qquad + ((-4x^2 - 2x + 1)y^3 + (x^3 - x^2)y^2)z^9 + ((2x^3 + x^2)y^2 - 4xy^3)z^8 \\\nonumber
		& \qquad + ((-2x - 2)y^3 + 6x^2y^2)z^7 + ((2x^2 + 4x - 1)y^2 - y^3)z^6 \\\nonumber
		& \qquad + ((6x - 1)y^2 - 2x^2y)z^5 + ((2x + 2)y^2 - x^2y)z^4 \\\nonumber
		& \qquad + (3y^2 - y - x^2)z^3 + y^2z^2 + yz + 1, \\\nonumber
		q(x, y, z) & = x^6y^3z^{17} + ((x^5 + x^4)y^3 + x^6y^2)z^{15} + 4x^4y^3z^{14} + (x^3y^3 + x^5y^2)z^{13} \\\nonumber
		& \qquad + ((3x^3 + 3x^2)y^3 + 2x^4y^2)z^{12} + (4x^2y^3 - x^4y^2)z^{11} + (2xy^3 + x^3y^2)z^{10} \\\nonumber
		& \qquad + ((2x + 1)y^3 + (-x^3 - 2x^2)y^2)z^9 + (y^3 - 3x^2y^2)z^8 + (-x^2 - 2x)y^2z^7 \\\nonumber
		& \qquad + (-2x - 1)y^2z^6  + (-2y^2 - 2x^2y)z^5 - y^2z^4 + (-y - x^2)z^3 + 1,
		\end{align}}%
	where $x$ is a formal variable associated with the number of one-story houses in the configuration, $y$ is a formal variable associated with the number of two-story houses in the configuration and $z$ is a formal variable associated with the length of the configuration.\\
	\item  Trivariate generating function for the mixed Flory model is given by
	{\tiny \begin{align}\label{it:mixed_flory}
		F(x, y, z) & = -\frac{p(x, y, z)}{q(x, y, z)},\\\nonumber
		p(x, y, z) & = x^2yz^8 + x^2yz^7 - xyz^6 - 2xyz^5 + ((1 - 2x)y + 2x^2)z^4 \\\nonumber
		& \qquad + ((2 - x)y + x^2)z^3 + (2y - x)z^2 + yz + 1, \\\nonumber
		q(x, y, z) & = x^2yz^9 - xyz^7 - xyz^6 + ((1 - x)y + x^2)z^5 + yz^4 + yz^3 + xz^2 - 1,
		\end{align}}%
	where $x$ is a formal variable associated with the number of one-story houses in the configuration, $y$ is a formal variable associated with the number of two-story houses in the configuration and $z$ is a formal variable associated with the length of the configuration.\\
	\item  Bivariate generating function for the $2 \times n$ model is given by
	{\tiny \begin{align}\label{it:2xn}
		F(x, y) & = -\frac{p(x, y)}{q(x, y)}, \\\nonumber
		p(x, y) & = x^8y^5 - (x^5 + x^4)y^3 + (2x^3 - x^4)y^2 + (x - x^2)y - 1, \\\nonumber
		q(x, y) & = x^9y^6 - x^6y^4 + (x^4 - x^5)y^3 - x^3y^2 - xy + 1,
		\end{align}
	}%
	where $x$ is a formal variable associated with the number of houses in the configuration and $y$ is a formal variable associated with the length of the configuration.\\
	\item  Bivariate generating function for the $3 \times n$ model is given by
	{\tiny \begin{align}\label{it:3xn}
		F(x, y) & = -\frac{p(x, y)}{q(x, y)}, \\\nonumber
		p(x, y) & = (x^{40} - x^{39})y^{19} + (2x^{38} - x^{37})y^{18} + (-2x^{37} + 3x^{36} - 2x^{35} + x^{34} - x^{33})y^{17} \\\nonumber
		& \qquad + (-5x^{35} + x^{34} + 4x^{33} - x^{32} - x^{31})y^{16} + (-3x^{33} - x^{32} + 5x^{31} - x^{30} - 2x^{29})y^{15} \\\nonumber
		& \qquad + (4x^{31} - 9x^{30} + 12x^{29} - 7x^{28} + x^{26})y^{14} \\\nonumber
		& \qquad + (15x^{29} - 13x^{28} + 4x^{27} - 3x^{26} - 2x^{25} + x^{24})y^{13} + (9x^{27} - 9x^{26} + 3x^{25} + 2x^{22})y^{12} \\\nonumber
		& \qquad + (11x^{24} - 11x^{23} + 12x^{22} - 5x^{21})y^{11} + (-15x^{23} + 34x^{22} - 11x^{21} + 9x^{20} - 7x^{19})y^{10} \\\nonumber
		& \qquad + (-9x^{21} + 30x^{20} - 14x^{19} + 7x^{18} - 9x^{17} + x^{16})y^9 \\\nonumber
		& \qquad + (-4x^{19} + 14x^{18} + 4x^{16} - x^{14} + x^{13})y^8 \\\nonumber
		& \qquad + (5x^{17} - 23x^{16} + 10x^{15} + x^{14} + x^{13} - x^{12})y^7 \\\nonumber
		& \qquad + (3x^{15} - 21x^{14} + 20x^{13} - 2x^{12} + x^{11} - x^{10} - x^9)y^6 \\\nonumber
		& \qquad + (2x^{13} - 15x^{12} + 15x^{11} + 2x^{10})y^5 + (-3x^{10} - 3x^9 + x^8 - x^7 + x^6)y^4 \\\nonumber
		& \qquad + (-x^8 - 5x^7 + x^6 + x^5)y^3 + (-x^6 + x^4 + x^3)y^2 - x^3y - 1, \\\nonumber
		q(x, y) & = 
		(x^{41} - x^{40})y^{20} + (2x^{39} - x^{38})y^{19} + (-2x^{38} + 3x^{37} - 2x^{36})y^{18} \\\nonumber
		& \qquad + (-5x^{36} + x^{35} + 3x^{34} - 2x^{33})y^{17} + (-3x^{34} - x^{33} + 5x^{32} - 2x^{31})y^{16} \\\nonumber
		& \qquad + (4x^{32} - 7x^{31} + 10x^{30} - 4x^{29})y^{15} + (15x^{30} - 8x^{29} + 4x^{28} - x^{26})y^{14} \\\nonumber
		& \qquad + (9x^{28} - 6x^{27} - 2x^{26} + x^{25} - x^{24})y^{13} + (7x^{25} - 11x^{24} + 7x^{23} - 2x^{22})y^{12} \\\nonumber
		& \qquad + (-15x^{24} + 19x^{23} - 12x^{22} + 5x^{21} - 3x^{20})y^{11} + (-9x^{22} + 21x^{21} - 6x^{20} + 6x^{19})y^{10} \\\nonumber
		& \qquad + (-4x^{20} + 14x^{19} + 2x^{18} + 6x^{17} + 3x^{16})y^9 \\\nonumber
		& \qquad  + (5x^{18} - 8x^{17} + 6x^{16} - x^{15} + x^{14} - x^{13})y^8 + (3x^{16} - 12x^{15} + 4x^{14} - x^{13})y^7 \\\nonumber
		& \qquad  + (2x^{14} - 11x^{13} - x^{12} + 3x^{11} + x^{10} + x^9)y^6 + (-8x^{11} - 6x^{10} + x^9)y^5 \\\nonumber
		& \qquad + (-4x^9 - 4x^8 + x^7 - x^6)y^4 + (-3x^7 - x^5)y^3 + (-x^4 - x^3)y^2 + 1,
		\end{align}}%
	where $x$ is a formal variable associated with the number of houses in the configuration and $y$ is a formal variable associated with the length of the configuration.
\end{itemize}

\end{document}